\documentclass[final,onefignum,onetabnum]{siamart171218}
\usepackage{amssymb}

\makeatletter

\numberwithin{equation}{section}

\newsiamremark{remark}{Remark}
\newsiamremark{problem}{Problem}
\newsiamremark{claim}{Claim}
\newsiamremark{example}{Example}

\newcommand{\dom}{\mbox{\rm dom}}

\makeatother

\title{Distributed deterministic asynchronous algorithms in time-varying
graphs through Dykstra splitting\thanks{Submitted to the editors DATE.
\funding{Grant R-146-000-214-112 from the Faculty of Science, National University of Singapore}}}
\author{C.H. Jeffrey Pang \thanks{Department of Mathematics, National University of Singapore (\email{matpchj@nus.edu.sg},\url{http://www.math.nus.edu.sg/\~matpchj}).}}
\ifpdf
\hypersetup{ pdftitle={Distributed deterministic asynchronous algorithms in time-varying
graphs through Dykstra splitting},
  pdfauthor={C.H. Jeffrey Pang} }
\fi

\headers{Distributed Dykstra splitting}{C.H. Jeffrey Pang}
\begin{document}
\maketitle

\begin{abstract}
Consider the setting where each vertex of a graph has a function,
and communications can only occur between vertices connected by an
edge. We wish to minimize the sum of these functions. For the case
when each function is the sum of a strongly convex quadratic and a
convex function, we propose a distributed version of Dykstra's algorithm.
The computations to optimize the dual objective function can run asynchronously
without a global clock, and in a distributed manner without a central
controller. Convergence to the primal minimizer is deterministic instead
of being probabilistic, and is guaranteed as long as in each cycle,
the edges where two-way communications occur connects all vertices.
We also look at an accelerated algorithm, and an algorithm for the
case when the functions on the nodes are not strongly convex.
\end{abstract}



\section{Introduction}

Let $\mathcal{G}=(\mathcal{V},\mathcal{E})$ be an undirected and
connected graph defined by the set of nodes (agents) $\mathcal{V}$
and the set of edges $\mathcal{E}\subset\mathcal{V}\times\mathcal{V}$.
Since \emph{$\mathcal{G}$} is undirected, we assume that both $(i,j)$
and $(j,i)$ refer to the same edge when it exists. 

Let $X$ be a finite dimensional Hilbert space. For a closed convex
set $C$, let $\delta_{C}(\cdot)$ be the indicator function defined
as $\delta_{C}(x)=0$ if $x\in C$, and equals $\infty$ otherwise.
For each edge $(i,j)\in\mathcal{E}$, let the hyperplane $H_{(i,j)}\subset X^{|\mathcal{V}|}$
be defined by 
\begin{equation}
H_{(i,j)}=\{(x_{1},\dots,x_{n}):x_{i}=x_{j}\}.\label{eq:H-i-j-subspace}
\end{equation}

We consider the following problem throughout the rest of this paper.

\begin{problem}

\label{prob:prob-statement} Let $(\mathcal{V},\mathcal{E})$ be a
connected graph. Suppose $H_{(i,j)}$ is defined as in \eqref{eq:H-i-j-subspace}
for all $(i,j)\in\mathcal{E}$, and let $f_{i}:X\to\bar{\mathbb{R}}$
(where $\bar{\mathbb{R}}:=\mathbb{R}\cup\{\infty\}$ throughout this
paper) be closed convex functions for all $i\in\mathcal{V}$. Let
$\mathbf{f}_{i}:X^{|\mathcal{V}|}\to\bar{\mathbb{R}}$ be defined
by $\mathbf{f}_{i}(x)=f_{i}(x_{i})$ (i.e., $\mathbf{f}_{i}$ depends
only on $i$-th variable). The primal problem of interest is 
\begin{equation}
\min_{x\in X^{|\mathcal{V}|}}\sum_{(i,j)\in\mathcal{E}}\underbrace{\delta_{H_{(i,j)}}(x)}_{h_{(i,j)}(x)}+\sum_{i\in\mathcal{V}}\underbrace{\mathbf{f}_{i}(x)}_{h_{i}(x)}.\label{eq:common-primal}
\end{equation}

\end{problem}

For each $\alpha\in\mathcal{E}\cup\mathcal{V}$, the function $h_{\alpha}:X^{|\mathcal{V}|}\to\bar{\mathbb{R}}$
is as marked in \eqref{eq:common-primal}. Since $(\mathcal{V},\mathcal{E})$
is connected, the problem \eqref{eq:common-primal} is equivalent
to 
\begin{equation}
\min_{x\in X}\sum_{i\in\mathcal{V}}f_{i}(x),\label{eq:centralized-primal}
\end{equation}
but we write it in the form \eqref{eq:common-primal} to emphasize
that the only vertex which has knowledge of the function $f_{i}(\cdot)$
is the vertex $i$. 

\subsection{\label{subsec:Distrib-algs}Distributed algorithms for \eqref{eq:common-primal}}

We give a brief summary of distributed algorithms for minimizing \eqref{eq:common-primal}.
Some properties desirable for a distributed algorithm, especially
when $|\mathcal{V}|$ is large, are as follows:
\begin{enumerate}
\item The algorithm is applicable to directed graphs, where only one way
communication is allowed between two vertices connected by a directed
edge. 
\item The algorithm has deterministic convergence.
\item The algorithm is asynchronous. There is no need for a global clock,
and each node can perform calculations at its own pace without being
affected by other slower nodes.
\item The algorithm is distributed (i.e., in intermediate computations,
each node only exchanges data with its neighbors) and decentralized
(i.e., there is no central node connected to all other nodes to coordinate
computations).
\item The algorithm allows for time-varying graphs.
\end{enumerate}
We emphasize that the algorithm that we look at in this paper is only
applicable to undirected graphs, and hence does not satisfy property
(1). Nevertheless, we give a brief summary of the literature behind
distributed algorithms for directed graphs in this paragraph. In the
case where only one way communication is allowed between two vertices
connected by a directed edge, the survey \cite{Nedich_survey} records
many algorithms derived from the subgradient algorithm for solving
\eqref{eq:common-primal}. If the edges in a network are directed,
it appears that the subgradient method is the only reasonable method.
The subgradient method requires diminishing step sizes for convergence
in the general case, which affects its convergence rates. More details
of recent developments are in \cite{Nedich_talk_2017}. A notable
paper is \cite{EXTRA_Shi_Ling_Wu_Yin}. The case of time-varying graphs
was first studied in \cite{Nedich_Olshevsky} and further extended
in \cite{Nedich_Olshevsky_Shi}. In time-varying graphs, the assumption
needed for convergence is for the edge set $\mathcal{E}$ to vary
over time. But if the edges are undirected, then alternative methods
may be possible, and would usually be faster than subgradient methods.
For strongly convex problems, linear convergence is possible. These
algorithms appear to be synchronous, and require the functions involved
to be smooth.

Two common methods for minimizing the sum of two convex functions
are the ADMM and Peaceman-Rachford algorithms (with the Douglas Rachford
algorithm a special case of the latter). The Peaceman-Rachford algorithm
is an example of a splitting method, and it is well known that the
ADMM is dual to the Douglas Rachford method \cite{Gabay_ADMM}. In
order to minimize the sum of more than two functions, the product
space reformulation is a well-studied option. (See for example \cite[Chapter 7]{Boyd_Eckstein_ADMM_review}.)
Another strategy is \cite{Eckstein_Svaiter_2009_SICON,Eckstein_Combettes_MAPR},
which is a splitting method for the sum of more than two functions
without using the product space reformulation. The latter development
in \cite{Eckstein_Combettes_MAPR} allows for lags in the collection
of data for nodes where the computation time is greater, thus allowing
for an asynchronous operation. Still, this algorithm requires a central
controller, so it is different from the algorithms we consider in
this paper. 

We now look at asynchronous distributed algorithm with deterministic
convergence (rather than probabilistic convergence). In some applications,
the guarantees from deterministic convergence can outweigh other advantages
of algorithms with randomized convergence. We mention that the paper
\cite{Gurbuzbalaban_Ozdaglar_Parrilo_SIOPT_2017} and the extension
\cite{Aytekin_F_Johansson_2016} are algorithms that give deterministic
convergence for strongly convex problems that are primal in nature,
so these algorithms cannot handle more than one constraint sets. The
method in \cite{Aybat_Hamedani_2016} may arguably be considered to
have these properties. Other than that, we are not aware of a decentralized,
asynchronous algorithm that has deterministic convergence for \eqref{eq:common-primal}
and is not a subgradient method. 

But a decentralized asynchronous probabilistic method derived from
the ADMM is proposed in \cite{Iutzeler_Bianchi_Ciblat_Hachem_1st_paper_dist,Bianchi_Hachem_Iutzeler_2nd_paper_dist}.
The key idea in the first paper is the introduction of a randomized
Gauss-Seidel iterations of the so called Douglas-Rachford operator,
and the second paper extends the first by incorporating the work of
\cite{Vu_2013,Condat_2013}. This concept was generalized in \cite{AROCK_Peng_Xu_Yan_Yin}.
All the works just mentioned use monotone operator theory (see for
example the textbook \cite{BauschkeCombettes11}). Such algorithms
require computations in the nodes to follow specific probability distributions,
so they do not seem immediately applicable to the setting of time-varying
graphs in \cite{Nedich_Olshevsky} mentioned earlier. Another randomized,
distributed method for nonconvex functions is \cite{Facchinei_Scutari_2016}.

Another randomized distributed method most similar to what we discuss
in this paper is that in \cite{Notars_asyn_distrib_2015}. We discuss
this more in Subsection \ref{subsec:Dyk-method}.

As mentioned in the survey \cite{Nedich_survey}, the primal problem
\eqref{eq:centralized-primal} has dual 
\begin{eqnarray}
 & \max_{\{y_{i}\}_{i\in\mathcal{V}}\subset X} & -\sum_{i\in\mathcal{V}}f_{i}^{*}(y_{i})\label{eq:better-dual}\\
 & \mbox{s.t.} & \sum_{i\in\mathcal{V}}y_{i}=0,\nonumber 
\end{eqnarray}
which is also known as the resource allocation problem. 

\subsection{\label{subsec:Dyk-method}A special case of \eqref{eq:common-primal}
through Dykstra's algorithm}

When some of the functions in the primal problem of the form \eqref{eq:centralized-primal}
are extended valued, it may be difficult to find a primal feasible
point in the first place. We first look at the problem 

\begin{equation}
\min_{x\in X^{|\mathcal{V}|}}\frac{1}{2}\|x-x_{0}\|^{2}+\sum_{(i,j)\in\mathcal{E}}\delta_{H_{(i,j)}}(x)+\sum_{i\in\mathcal{V}}\mathbf{f}_{i}(x),\label{eq:Dyk-primal}
\end{equation}
where $\|x\|^{2}:=\langle x,x\rangle=\sum_{i=1}^{n}\langle x_{i},x_{i}\rangle$
and $x_{0}\in X^{|\mathcal{V}|}$. This problem fits the framework
of \eqref{eq:common-primal} since 
\begin{equation}
\frac{1}{2}\|x-x_{0}\|^{2}+\sum_{i\in\mathcal{V}}\mathbf{f}_{i}(x)=\sum_{i\in\mathcal{V}}\left[\mathbf{f}_{i}(x)+\frac{1}{2}\|x_{i}-[x_{0}]_{i}\|^{2}\right].\label{eq:separate_obj}
\end{equation}

In the case when $\mathbf{f}_{i}\equiv0$ for all $i\in\mathcal{V}$,
the problem reduces to the average consensus problem \cite{Boyd_distrib_averaging,Distrib_averaging_Dimakis_Kar_Moura_Rabbat_Scaglione}.
The minimizer to \eqref{eq:Dyk-primal} is the vector $(\bar{x},\dots,\bar{x})\in X^{|\mathcal{V}|}$,
where $\bar{x}=\frac{1}{|\mathcal{V}|}\sum_{i\in\mathcal{V}}[x_{0}]_{i}$. 

The (Fenchel) dual to \eqref{eq:Dyk-primal} is

\begin{equation}
\max_{z_{\alpha}\in X^{|\mathcal{V}|},\alpha\in\mathcal{E}\cup\mathcal{V}}F(\{z_{\alpha}\}_{\alpha\in\mathcal{E}\cup\mathcal{V}}),\label{eq:dual-fn}
\end{equation}
where 
\begin{eqnarray}
 &  & F(\{z_{\alpha}\}_{\alpha\in\mathcal{E}\cup\mathcal{V}})\nonumber \\
 & := & -\frac{1}{2}\left\Vert x_{0}-\sum_{\alpha\in\mathcal{E}\cup\mathcal{V}}z_{\alpha}\right\Vert ^{2}+\frac{1}{2}\|x_{0}\|^{2}-\sum_{(i,j)\in\mathcal{E}}\delta_{H_{(i,j)}}^{*}(z_{(i,j)})-\sum_{i\in\mathcal{V}}\mathbf{f}_{i}^{*}(z_{i}).\label{eq:Dykstra-dual-defn}
\end{eqnarray}

In the case when $f_{i}(\cdot)\equiv\delta_{C_{i}}(\cdot)$ for some
closed convex set $C_{i}$ for all $i$, Dykstra's algorithm finds
the primal minimizer of the problem 
\begin{equation}
\min_{x\in X}\frac{1}{2}\|x-x_{0}\|^{2}+\sum_{i=1}^{n}f_{i}(x),\label{eq:illus-dyk-P}
\end{equation}
where $x_{0}\in X$ (note that in \eqref{eq:illus-dyk-P}, $x$ and
$x_{0}$ lie in $X$ instead of $X^{|\mathcal{V}|}$ like in \eqref{eq:Dyk-primal})
by maximizing the dual 
\begin{equation}
\max_{y\in X^{n}}-\frac{1}{2}\left\Vert x_{0}-\sum_{i=1}^{n}y_{i}\right\Vert -\sum_{i=1}^{n}f_{i}^{*}(y_{i})+\frac{1}{2}\|x_{0}\|^{2}\label{eq:illus-dyk-D}
\end{equation}
through block coordinate minimization. If each $f_{i}(\cdot)$ are
allowed to be any closed convex function, it can now be seen that
\eqref{eq:Dykstra-dual-defn} is actually a special case of \eqref{eq:illus-dyk-D}.

Dykstra's algorithm was first studied in \cite{Dykstra83} in the
case where $f_{i}(\cdot)\equiv\delta_{C_{i}}(\cdot)$ and $C_{i}$
are closed convex sets for all $i$. The convergence of the primal
iterates to the projection of $x_{0}$ onto $\cap_{i=1}^{n}C_{i}$
was proved in \cite{BD86}, and is sometimes called the Boyle-Dykstra
theorem. Dykstra's algorithm was independently noted in \cite{Han88}
to be block coordinate minimization on the dual problem. The proof
in \cite{BD86} was adapted in \cite{Gaffke_Mathar} using duality.
We remark that the Boyle-Dykstra theorem is remarkable because the
convergence to the primal minimizer occurs even when there is no dual
optimizer. (For example, look at \cite[page 9]{Han88} where two circles
in $\mathbb{R}^{2}$ intersect at only one point.) The case when sampling
of the sets is noncyclic is addressed in \cite{Hundal-Deutsch-97}
(among other things not directly relevant to this paper). As pointed
out in \cite{Pang_Dyk_spl}, the Boyle-Dykstra theorem holds even
if $f_{i}(\cdot)$ are closed convex functions instead of $\delta_{C_{i}}(\cdot)$.
 (We recently became aware that the dual ascent interpretation can
be traced to \cite{Combettes_Dung_Vu_SVAA,Combettes_Dung_Vu_JMAA,Pesquet_Abboud_gang_2017_JMIV},
but the connection to distributed optimization was not pointed out
there.) For more on the background on Dykstra's algorithm, we refer
to \cite{BauschkeCombettes11,BB96_survey,Deustch01,EsRa11}. Some
recent work on Dykstra's algorithm include \cite{Tibshirani_on_Dykstra}.

Dykstra's algorithm was extended to a distributed algorithm in \cite{Borkar_distrib_dyk},
and they highlight the works \cite{Aybat_Hamedani_2016,LeeNedich2013,Nedic_Ram_Veeravalli_2010,Ozdaglar_Nedich_Parrilo}
on distributed optimization. The work in \cite{Borkar_distrib_dyk}
is vastly different from how Dykstra's algorithm is studied in \cite{BD86}
and \cite{Gaffke_Mathar}.

It turns out that \cite{Notars_asyn_distrib_2015} discusses a similar
problem to \eqref{eq:Dyk-primal}. They generalize \eqref{eq:Dyk-primal}
by allowing the functions $x_{i}\mapsto\frac{1}{2}\|x_{i}-[x_{0}]_{i}\|^{2}$
to be any strongly convex function, and proceed to calculate that
the dual has a similar form as \eqref{eq:dual-fn} and \eqref{eq:Dykstra-dual-defn}.
Their dual is still a sum of a smooth component and a separable component,
which they solve with randomized dual proximal gradient. We discuss
the differences between their paper and ours in Subsection \ref{subsec:contrib}.

\subsection{\label{subsec:contrib}Contributions of this paper}

In this paper, we propose looking at the formulation \eqref{eq:Dyk-primal}
and show that Dykstra's algorithm applied to this formulation gives
an algorithm with properties (2)-(5) in Subsection \ref{subsec:Distrib-algs}.
As stated in the introduction, we are not aware of any other asynchronous
distributed algorithm that has deterministic convergence other than
\cite{Gurbuzbalaban_Ozdaglar_Parrilo_SIOPT_2017,Aytekin_F_Johansson_2016,Aybat_Hamedani_2016},
though our assumption that the functions on each vertex has a strongly
convex function with known modulus might be a bit strong.

We highlight the differences from \cite{Notars_asyn_distrib_2015}.
The first difference is that we show that Dykstra's algorithm gives
deterministic convergence (property (2)), whereas \cite{Notars_asyn_distrib_2015}
pointed out probabilistic convergence. A naive application of Dykstra's
algorithm to \eqref{eq:Dyk-primal} would mean that all the edges
in the graph have to be used in one cycle, which would not cover the
setting of time-varying graphs as done in \cite{Nedich_Olshevsky}.
But we show that as long as the graph is (using the definition in
\cite{Nedich_Olshevsky}) uniformly connected, then convergence can
be achieved. (See Remark \ref{rem:time-vary}.) Dykstra's splitting
also gives these two desirable properties that were not noticed in
\cite{Notars_asyn_distrib_2015}:
\begin{enumerate}
\item [(6)]The iterates of the algorithm converges to the primal minimizer
even when a dual minimizer does not exist.
\item [(7)]Since Dykstra's splitting is a dual ascent algorithm, as many
dual variables can be maximized at one time as possible. This is an
advantage as subproblems involving more dual variables lead to greedier,
and possibly greater, increase of the dual objective value.
\end{enumerate}
Next, in Section \ref{sec:non-strongly-convex-alg}, we look at a
decentralized dual ascent algorithm for \eqref{eq:common-primal}
(which does not have the quadratic term) through the dual problem
\eqref{eq:better-dual}. Once again, the algorithm is asynchronous.
In contrast to the adapted Dykstra's algorithm, we now optimize dual
variables corresponding to a collection of vertices at a time. We
show an example where convergence fails. The algorithm works on collections
of vertices of the graph at a time, and is thus robust to lost communications
in edges of the graph.

Lastly, in Section \ref{sec:Accelerate}, we discuss ideas for an
accelerated proximal gradient method on the dual \eqref{eq:dual-fn}.
This algorithm runs on a global clock, and it does not work for time-varying
graphs. (I.e., it does not satisfy properties (4)-(5)). But asynchronous
greedy steps satisfying property (7) can be performed to speed up
the increase of the dual objective value. 

\section{\label{sec:conv-Dyk}Convergence of distributed Dykstra's algorithm}

In this section, we state our distributed Dykstra's algorithm, make
some remarks that may be helpful in understanding the algorithm, and
prove its convergence without constraint qualifications.

\subsection{Statement of distributed Dykstra's algorithm}

Let $D\subset X^{|\mathcal{V}|}$ be the diagonal set defined by 
\[
D:=\{x\in X^{|\mathcal{V}|}:x_{1}=x_{2}=\cdots=x_{|\mathcal{V}|}\}.
\]
With the definition of $H_{(i,j)}$ in \eqref{eq:H-i-j-subspace}
and $\mathcal{G}=(\mathcal{V},\mathcal{E})$ being a connected graph,
it is obvious that 
\begin{equation}
\bigcap_{(i,j)\in\mathcal{E}}H_{(i,j)}=D\mbox{ and }\sum_{(i,j)\in\mathcal{E}}H_{(i,j)}^{\perp}=D^{\perp}=\left\{ z\in X^{|\mathcal{V}|}:\sum_{i\in\mathcal{V}}z_{i}=0\right\} .\label{eq:D-and-D-perp}
\end{equation}

\begin{proposition}
\label{prop:E-connects-V}Suppose $\mathcal{G}=(\mathcal{V},\mathcal{E})$
is a connected graph. Let $H_{(i,j)}$ be the set \eqref{eq:H-i-j-subspace}
(defining the linear constraints relating the connection between nodes
$i$ and $j$). Let $\mathcal{E}'$ be a subset of $\mathcal{E}$.
The following conditions are equivalent:

\begin{enumerate}
\item $\cap_{(i,j)\in\mathcal{E}'}H_{(i,j)}=D$
\item $\sum_{(i,j)\in\mathcal{E}'}H_{(i,j)}^{\perp}=D^{\perp}.$
\item The graph $\mathcal{G}'=(\mathcal{V},\mathcal{E}')$ is connected. 
\end{enumerate}
\end{proposition}

\begin{proof}
The equivalence between (1) and (3) is easy, and the equivalence between
(1) and (2) is simple linear algebra. 
\end{proof}
\begin{definition}
We say that\emph{ $\mathcal{E}'$ connects $\mathcal{V}$ }if any
of the equivalent properties in Proposition \ref{prop:E-connects-V}
is satisfied.
\end{definition}

We prove a lemma. 
\begin{lemma}
\label{lem:express-v-as-sum}(Expressing $v$ as a sum) Suppose $X$
is a finite dimensional Hilbert space. There is a $C_{1}>0$ such
that for all $v\in D^{\perp}$ and $\mathcal{E}'\subset\mathcal{E}$
such that $\mathcal{E}'$ connects $\mathcal{V}$, we can find $z_{(i,j)}\in H_{(i,j)}^{\perp}$
for all $(i,j)\in\mathcal{E}'$ such that $\sum_{(i,j)\in\mathcal{E}'}z_{(i,j)}=v$
and $\|z_{(i,j)}\|\leq C_{1}\|v\|$ for all $(i,j)\in\mathcal{E}'$.
\end{lemma}

\begin{proof}
This is elementary, so we only give an outline. Fix an $\mathcal{E}'$.
We can choose $\tilde{H}_{(i,j)}\subset H_{(i,j)}^{\perp}$ so that
$\sum_{(i,j)\in\mathcal{E}'}H_{(i,j)}^{\perp}=D^{\perp}$ is a direct
sum of $\{\tilde{H}_{(i,j)}\}_{(i,j)\in\mathcal{E}'}$. So $v$ can
be written uniquely as the sum $v=\sum_{(i,j)\in\mathcal{E}'}z_{(i,j)}$,
where $z_{(i,j)}\in\tilde{H}_{(i,j)}$. The mapping from $v$ to each
$z_{(i,j)}\in\tilde{H}_{(i,j)}$ is linear, and this linear map has
a norm bounded by some $C_{(i,j),\mathcal{E}'}$. Letting $C_{1}$
be the maximum of these $C_{(i,j),\mathcal{E}'}$ gives us our conclusion. 
\end{proof}
We present our distributed Dykstra's algorithm in Algorithm \ref{alg:Ext-Dyk}. 

\begin{algorithm}[H]
\caption{Decentralized Dykstra's algorithm} 

\label{alg:Ext-Dyk}

Consider the problem \eqref{eq:Dyk-primal} along with the associated
dual problem \eqref{eq:dual-fn}.

Let $\bar{w}$ be a positive integer. Let $C_{1}>0$ satisfy Lemma
\ref{lem:express-v-as-sum}. Our decentralized Dykstra's algorithm
is as follows:

01$\quad$Let 

\begin{itemize}
\item $z_{i}^{1,0}\in X^{|\mathcal{V}|}$ be a starting dual vector for
$\mathbf{f}_{i}(\cdot)$ for each $i\in\mathcal{V}$ so that $[z_{i}^{1,0}]_{j}=0$
for all $j\in\mathcal{V}\backslash\{i\}$. 
\begin{itemize}
\item $v_{H}^{1,0}\in D^{\perp}$ be a starting dual vector for \eqref{eq:dual-fn}.

\begin{itemize}
\item Note: $\{z_{(i,j)}^{n,0}\}_{(i,j)\in\mathcal{E}}$ is defined through
$v_{H}^{n,0}$ in \eqref{eq_m:resetted-z-i-j}.
\end{itemize}
\item Let $x^{1,0}$ be $x^{1,0}=x_{0}-v_{H}^{1,0}-\sum_{i\in\mathcal{V}}z_{i}^{1,0}$.
\end{itemize}
\end{itemize}
02$\quad$For $n=1,2,\dots$

03$\quad$$\quad$Let $\mathcal{E}_{n}\subset\mathcal{E}$ be such
that $\mathcal{E}_{n}$ connects $\mathcal{V}$. 

04$\quad$$\quad$Define $\{z_{(i,j)}^{n,0}\}_{(i,j)\in\mathcal{E}}$
so that:\begin{subequations}\label{eq_m:resetted-z-i-j} 
\begin{eqnarray}
z_{(i,j)}^{n,0} & = & 0\mbox{ for all }(i,j)\notin\mathcal{E}_{n}\label{eq:reset-z-i-j-1}\\
z_{(i,j)}^{n,0} & \in & H_{(i,j)}^{\perp}\mbox{ for all }(i,j)\in\mathcal{E}\label{eq:reset-z-i-j-2}\\
\|z_{(i,j)}^{n,0}\| & \leq & C_{1}\|v_{H}^{n,0}\|\mbox{ for all }(i,j)\in\mathcal{E}\label{eq:reset-z-i-j-3}\\
\mbox{ and }\sum_{(i,j)\in\mathcal{E}}z_{(i,j)}^{n,0} & = & v_{H}^{n,0}.\label{eq:reset-z-i-j-4}
\end{eqnarray}
\end{subequations}

$\quad$$\quad$(This is possible by Lemma \ref{lem:express-v-as-sum}.) 

05$\quad$$\quad$For $w=1,2,\dots,\bar{w}$

06$\quad$$\quad$$\quad$Choose a set $S_{n,w}\subset\mathcal{E}_{n}\cup\mathcal{V}$
such that $S_{n,w}\neq\emptyset$. 

07$\quad$$\quad$$\quad$Define $\{z_{\alpha}^{n,w}\}_{\alpha\in S_{n,w}}$
by 
\begin{equation}
\{z_{\alpha}^{n,w}\}_{\alpha\in S_{n,w}}=\underset{z_{\alpha},\alpha\in S_{n,w}}{\arg\min}\frac{1}{2}\left\Vert x_{0}-\sum_{\alpha\notin S_{n,w}}z_{\alpha}^{n,w-1}-\sum_{\alpha\in S_{n,w}}z_{\alpha}\right\Vert ^{2}+\sum_{\alpha\in S_{n,w}}h_{\alpha}^{*}(z_{\alpha}).\label{eq:Dykstra-min-subpblm}
\end{equation}

08$\quad$$\quad$$\quad$Set $z_{\alpha}^{n,w}:=z_{\alpha}^{n,w-1}$
for all $\alpha\notin S_{n,w}$.

09$\quad$$\quad$End For 

10$\quad$$\quad$Let $z_{i}^{n+1,0}=z_{i}^{n,\bar{w}}$ for all $i\in\mathcal{V}$
and $v_{H}^{n+1,0}=v_{H}^{n,\bar{w}}=\sum_{(i,j)\in\mathcal{E}}z_{(i,j)}^{n,\bar{w}}$.

11$\quad$End For 
\end{algorithm}
\begin{remark}

(Intuition behind Algorithm \ref{alg:Ext-Dyk}) We now provide some
intuition behind Algorithm \ref{alg:Ext-Dyk}. The classical Dykstra
splitting approach is the block coordinate maximization of the dual
problem \eqref{eq:dual-fn}-\eqref{eq:Dykstra-dual-defn}. This is
reflected in lines 6-8 of Algorithm \ref{alg:Ext-Dyk}. In order for
Algorithm \ref{alg:Ext-Dyk} to handle time-varying graphs, we choose
$\mathcal{E}_{n}\subset\mathcal{E}$ in line 3 so that $\mathcal{E}_{n}$
connects $\mathcal{V}$, and the problem 
\[
\max_{z_{\alpha}\in X^{|\mathcal{V}|},\alpha\in\mathcal{E}_{n}\cup\mathcal{V}}-\frac{1}{2}\left\Vert x_{0}-\sum_{\alpha\in\mathcal{E}_{n}\cup\mathcal{V}}z_{\alpha}\right\Vert ^{2}+\frac{1}{2}\|x_{0}\|^{2}-\sum_{(i,j)\in\mathcal{E}_{n}}\delta_{H_{(i,j)}}^{*}(z_{(i,j)})-\sum_{i\in\mathcal{V}}\mathbf{f}_{i}^{*}(z_{i})
\]
(note the $\mathcal{E}_{n}$ in the above formula) would have the
same optimal objective value as \eqref{eq:dual-fn}-\eqref{eq:Dykstra-dual-defn}
since the corresponding primal problems are equivalent and strong
duality holds. The subset $\mathcal{E}_{n}\subset\mathcal{E}$ chosen
in line 3 may be such that $z_{(i,j)}\neq0$, but $(i,j)\notin\mathcal{E}_{n}$.
So we perform line 4 so that \eqref{eq_m:resetted-z-i-j} holds, which
implies that $z_{(i,j)}=0$ for all $(i,j)\notin\mathcal{E}_{n}$,
while perserving $v_{H}^{n,\bar{w}}=v_{H}^{n+1,0}$ (see \eqref{eq_m:all_acronyms}
later). As we shall see in Remark \ref{rem:irrelevance-of-z-i-j}
later, the reassignment of $\{z_{(i,j)}\}_{(i,j)\in\mathcal{E}}$
in line 4 is necessary for further analysis, but may be ignored in
implementing the algorithm. Algorithm \ref{alg:Ext-Dyk-1} then shows
an equivalent formulation of Algorithm \ref{alg:Ext-Dyk} where one
only keeps track of $x^{n,w}$ and $\{z_{i}^{n,w}\}_{i\in\mathcal{V}}$.
If $f_{i}(\cdot)\equiv0$, then $z_{i}^{n,w}$ is always $0$ for
all $i\in\mathcal{V}$, so Algorithm \ref{alg:Ext-Dyk-1} reduces
to the averaged consensus algorithm \cite{Boyd_distrib_averaging,Distrib_averaging_Dimakis_Kar_Moura_Rabbat_Scaglione}.

\end{remark}

\begin{remark}

(Choice of $S_{n,w}$) The choice of $S_{n,w}$ allows for a flexibility
in how large one wants the subproblem \eqref{eq:Dykstra-min-subpblm}
to be. It is easy to see that a small $S_{n,w}$ allows for the subproblems
to be small and easy to solve. The larger the size of $S_{n,w}$,
the harder the subproblem, but greater increase in the dual objective
value is expected. An issue of choosing large $|S_{n,w}|$ is that
we need an extra coordination between the nodes in $V'_{n,w}$. When
$S_{n,w}\subset\mathcal{V}$ and $|S_{n,w}|=1$, there is no coordination
needed. Similarly, when $S_{n,w}\subset\mathcal{E}$ and $|S_{n,w}|=1$,
only two nodes need to coordinate with each other, which is okay for
an undirected graph. An implementer can, for example, choose a star-like
subgraph (i.e., there is a central node in the subgraph connecting
to all others) and apply an algorithm suitable for problems with a
centralized node. Since the dual objective value acts as a Lyapunov
function, one could choose $S_{n,w}$ to be as large as one can reasonably
solve to increase the dual objective value as much as one can. This
increase in the dual objective value can be greater if large subproblems
are solved partially compared to small subproblems solved fully. 

\end{remark}

To simplify calculations, we let $v_{A}$, $v_{H}$ and $x$ be denoted
by\begin{subequations}\label{eq_m:all_acronyms} 
\begin{eqnarray}
v_{H} & = & \sum_{(i,j)\in\mathcal{E}}z_{(i,j)}\label{eq:v-H-def}\\
v_{A} & = & v_{H}+\sum_{i\in\mathcal{V}}z_{i}\label{eq:from-10}\\
x & = & x_{0}-v_{A}.\label{eq:x-from-v-A}
\end{eqnarray}
\end{subequations}Intuitively, $v_{H}$ describes the sum of the
dual variables due to $H_{(i,j)}$ for all $(i,j)\in\mathcal{E}$,
$v_{A}$ is the sum of all dual variables, and $x$ is the estimate
of the primal variable. 

The following inequality describes the duality gap between \eqref{eq:Dyk-primal}
and \eqref{eq:dual-fn}. 
\begin{eqnarray}
 &  & \begin{array}{c}
\frac{1}{2}\|x_{0}-x\|^{2}+\underset{\alpha\in\mathcal{E}\cup\mathcal{V}}{\sum}h_{\alpha}(x)-F(\{z_{\alpha}\}_{\alpha\in\mathcal{E}\cup\mathcal{V}})\end{array}\label{eq:From-8}\\
 & \overset{\eqref{eq:Dykstra-dual-defn}}{=} & \begin{array}{c}
\frac{1}{2}\|x_{0}-x\|^{2}+\underset{\alpha\in\mathcal{E}\cup\mathcal{V}}{\sum}[h_{\alpha}(x)+h_{\alpha}^{*}(z_{\alpha})]\end{array}\nonumber \\
 &  & \begin{array}{c}
\qquad-\left\langle x_{0},\underset{\alpha\in\mathcal{E}\cup\mathcal{V}}{\sum}z_{\alpha}\right\rangle +\frac{1}{2}\left\Vert \underset{\alpha\in\mathcal{E}\cup\mathcal{V}}{\sum}z_{\alpha}\right\Vert ^{2}\end{array}\nonumber \\
 & \overset{\scriptsize\mbox{Fenchel duality}}{\geq} & \begin{array}{c}
\frac{1}{2}\|x_{0}-x\|^{2}+\left\langle x,\underset{\alpha\in\mathcal{E}\cup\mathcal{V}}{\sum}z_{\alpha}\right\rangle -\left\langle x_{0},\underset{\alpha\in\mathcal{E}\cup\mathcal{V}}{\sum}z_{\alpha}\right\rangle +\frac{1}{2}\left\Vert \underset{\alpha\in\mathcal{E}\cup\mathcal{V}}{\sum}z_{\alpha}\right\Vert ^{2}\end{array}\nonumber \\
 & = & \begin{array}{c}
\frac{1}{2}\left\Vert x_{0}-x-\underset{\alpha\in\mathcal{E}\cup\mathcal{V}}{\sum}z_{\alpha}\right\Vert ^{2}\geq0.\end{array}\nonumber 
\end{eqnarray}

\begin{claim}

\label{claim:Fenchel-duality} In Algorithm \ref{alg:Ext-Dyk}, for
all $\alpha\in S_{n,w}$, we have 
\begin{enumerate}
\item [(a)]$-x^{n,w}+\partial h_{\alpha}^{*}(z_{\alpha}^{n,w})\ni0$,
\item [(b)]$-z_{\alpha}^{n,w}+\partial h_{\alpha}(x^{n,w})\ni0$, and
\item [(c)]$h_{\alpha}(x^{n,w})+h_{\alpha}^{*}(z_{\alpha}^{n,w})=\langle x^{n,w},z_{\alpha}^{n,w}\rangle$. 
\end{enumerate}
\end{claim}
\begin{proof}
By taking the optimality conditions in \eqref{eq:Dykstra-min-subpblm}
with respect to $z_{\alpha}$ for $\alpha\in S_{n,w}$ and making
use of \eqref{eq_m:all_acronyms} to get $x^{n,w}=x_{0}-\sum_{\alpha\in\mathcal{V}\cup\mathcal{E}}z_{\alpha}^{n,w}$,
we deduce (a). The equivalence of (a), (b) and (c) is standard. 
\end{proof}
Even though Algorithm \ref{alg:Ext-Dyk} is described so that each
node $i\in\mathcal{V}$ and edge $(i,j)\in\mathcal{E}$ contains a
variable $z_{\alpha}\in X^{|\mathcal{V}|}$, the size of the variable
$z_{\alpha}$ that needs to be stored in each node and edge is small
due to sparsity.
\begin{proposition}
\label{prop:sparsity}(Sparsity of $z_{\alpha}$) We have $[z_{i}^{n,w}]_{j}=0$
for all $j\in\mathcal{V}\backslash\{i\}$, $n\geq1$ and $w\in\{0,1,\dots,\bar{w}\}$.
Similarly, $[z_{(i,j)}^{n,w}]_{k}=0$ for all $k\in\mathcal{V}\backslash\{i,j\}$,
$n\geq1$ and $w\in\{0,1,\dots,\bar{w}\}$.
\end{proposition}

\begin{proof}
The result for $z_{i}^{n,w}$ holds for $n=1$ and $w=0$. Claim \ref{claim:Fenchel-duality}(b)
shows that $z_{i}^{n,w}\in\partial\mathbf{f}_{i}(x^{n,w})$ for all
$i\in S_{n,w}$. Note that since $[\mathbf{f}_{i}(x)]_{j}=0$ for
all $j\in\mathcal{V}\backslash\{i\}$, $[\partial\mathbf{f}_{i}(x)]_{j}=0$
for all $j\in\mathcal{V}\backslash\{i\}$, which easily gives what
we need.

For all $(i,j)\in\mathcal{E}$ and $n\geq1$, the line \eqref{eq:reset-z-i-j-2}
implies that $z_{(i,j)}^{n,0}\in H_{(i,j)}^{\perp}$, and Claim \ref{claim:Fenchel-duality}(b)
implies that $z_{(i,j)}^{n,w}\in H_{(i,j)}^{\perp}$ for all $w\in\{1,\dots,\bar{w}\}$.
This implies the result at hand for $z_{(i,j)}^{n,w}$. 
\end{proof}
Dykstra's algorithm is traditionally written in terms of solving for
the primal variable $x$. For completeness, we show the equivalence
between \eqref{eq:Dykstra-min-subpblm} and the primal minimization
problem. The proof is easily extended from \cite[Proposition 2.4]{Pang_Dyk_spl}
(The duality between \eqref{eq:Dykstra-min-subpblm} and \eqref{eq:primal-subpblm}
can also be obtained by Fenchel duality.) 
\begin{proposition}
(On solving \eqref{eq:Dykstra-min-subpblm}) If a minimizer $\{z_{\alpha}^{n,w}\}_{\alpha\in S_{n,w}}$
for \eqref{eq:Dykstra-min-subpblm} exists, then the $x^{n,w}$ in
\eqref{eq:x-from-v-A} satisfies 
\begin{equation}
x^{n,w}=\begin{array}{c}
\underset{x\in X^{|\mathcal{V}|}}{\arg\min}\underset{\alpha\in S_{n,w}}{\sum}h_{\alpha}(x)+\frac{1}{2}\left\Vert x-\left(x_{0}-\underset{\alpha\notin S_{n,w}}{\sum}z_{\alpha}^{n,w}\right)\right\Vert ^{2}.\end{array}\label{eq:primal-subpblm}
\end{equation}
Conversely, if $x^{n,w}$ solves \eqref{eq:primal-subpblm} with the
dual variables $\{\tilde{z}_{\alpha}^{n,w}\}_{\alpha\in S_{n,w}}$
satisfying 
\begin{equation}
\begin{array}{c}
\tilde{z}_{\alpha}^{n,w}\in\partial h_{\alpha}(x^{n,w})\mbox{ and }x^{n,w}-x_{0}+\underset{\alpha\notin S_{n,w}}{\overset{\phantom{\alpha\notin S_{n,w}}}{\sum}}z_{\alpha}^{n,w}+\underset{\alpha\in S_{n,w}}{\sum}\tilde{z}_{\alpha}^{n,w}=0,\end{array}\label{eq:primal-optim-cond}
\end{equation}
then $\{\tilde{z}_{\alpha}^{n,w}\}_{\alpha\in S_{n,w}}$ solves \eqref{eq:Dykstra-min-subpblm}. 
\end{proposition}

\subsection{Examples of $S_{n,w}$}

In this subsection, we elaborate on how to solve \eqref{eq:primal-subpblm},
and show that Algorithm \ref{alg:Ext-Dyk} is an extension of the
average consensus algorithm.

For an $S_{n,w}$ such that $S_{n,w}\cap\mathcal{E}\neq\emptyset$,
define $\mathcal{V}'_{n,w}$ by 
\begin{equation}
\mathcal{V}'_{n,w}=\{\mbox{all vertices that are endpoints of some edge in }S_{n,w}\cap\mathcal{E}\}.\label{eq:V-prime}
\end{equation}
Suppose $S_{n,w}\cap\mathcal{E}$ is such that the subgraph $(\mathcal{V}'_{n,w},S_{n,w}\cap\mathcal{E})$
is a connected graph with no cycles, and $S_{n,w}\cap\mathcal{V}\subset\mathcal{V}_{n,w}'$.
Let $\tilde{y}\in X^{|\mathcal{V}|}$ be defined by 
\begin{equation}
\tilde{y}:=x_{0}-\sum_{\alpha\notin S_{n,w}}z_{\alpha}^{n,w}\overset{\scriptsize{\mbox{line 8}}}{=}x_{0}-\sum_{\alpha\notin S_{n,w}}z_{\alpha}^{n,w-1}\overset{\eqref{eq_m:all_acronyms}}{=}x^{n,w-1}+\sum_{\alpha\in S_{n,w}}z_{\alpha}^{n,w-1}.\label{eq:tilde-y}
\end{equation}
Then the primal minimization problem \eqref{eq:primal-subpblm} becomes
\begin{eqnarray}
x^{n,w} & = & \begin{array}{c}
\underset{x\in X^{|\mathcal{V}|}}{\arg\min}\underset{i\in S_{n,w}\cap\mathcal{V}}{\sum}h_{i}(x)+\underset{(i,j)\in S_{n,w}\cap\mathcal{E}}{\sum}h_{(i,j)}(x)\end{array}\nonumber \\
 &  & \begin{array}{c}
\qquad+\frac{1}{2}\left\Vert x-\left(x_{0}-\underset{\alpha\notin S_{n,w}}{\sum}z_{\alpha}^{n,w}\right)\right\Vert ^{2}\end{array}\nonumber \\
 & \overset{\eqref{eq:tilde-y}}{=} & \begin{array}{c}
\underset{x\in X^{|\mathcal{V}|}}{\arg\min}\underset{i\in S_{n,w}\cap\mathcal{V}}{\overset{\phantom{i\in S_{n,w}\cap\mathcal{V}}}{\sum}}h_{i}(x)+\underset{(i,j)\in S_{n,w}\cap\mathcal{E}}{\sum}h_{(i,j)}(x)+\frac{1}{2}\left\Vert x-\tilde{y}\right\Vert ^{2}.\end{array}\label{eq:primal-subpblm-special-case}
\end{eqnarray}
Recall that $h_{i}:X^{|\mathcal{V}|}\to\mathbb{R}$ is a function
whose output depends only on the $i$-th coordinate, where $i\in\mathcal{V}$.
If $x^{n,w}$ were to solve \eqref{eq:primal-subpblm-special-case},
then $h_{(i,j)}(x^{n,w})=\delta_{H_{(i,j)}}(x^{n,w})$ is finite for
all $(i,j)\in S_{n,w}\cap\mathcal{E}$, which shows that $x^{n,w}\in H_{(i,j)}$
for all $(i,j)\in S_{n,w}\cap\mathcal{E}$. This in turn means that
all the components of $x^{n,w}$ indexed by $\mathcal{V}'_{n,w}$
would need to have the same value. So the problem \eqref{eq:primal-subpblm-special-case}
can be reduced to one which optimizes over a variable in $X$ (instead
of $X^{|\mathcal{V}|}$), which, for all $i'\in\mathcal{V}'_{n,w}$,
takes the form 
\begin{eqnarray}
x_{i'}^{n,w} & \overset{\eqref{eq:primal-subpblm-special-case}}{=} & \underset{x\in X}{\arg\min}\sum_{i\in S_{n,w}\cap\mathcal{V}}f_{i}(x)+\frac{1}{2}\sum_{i\in\mathcal{V}'_{n,w}}\|x-\tilde{y}_{i}\|^{2}\label{eq:updated-x-n-w}\\
 & = & \underset{x\in X}{\arg\min}\sum_{i\in S_{n,w}\cap\mathcal{V}}f_{i}(x)+\frac{|\mathcal{V}'_{n,w}|}{2}\left\Vert x-\frac{1}{|\mathcal{V}'_{n,w}|}\sum_{i\in\mathcal{V}'_{n,w}}\tilde{y}_{i}\right\Vert ^{2}\nonumber 
\end{eqnarray}
where $f_{i}:X\to\bar{\mathbb{R}}$ is defined as in Problem \ref{prob:prob-statement}.
The iterate $x^{n,w}\in X^{|\mathcal{V}|}$ can be expressed in terms
of $x^{n,w-1}$ via 
\begin{equation}
x_{i'}^{n,w}=\begin{cases}
x_{i'}^{n,w-1} & \mbox{ if }i'\notin\mathcal{V}'_{n,w}\\
\mbox{The formula in \eqref{eq:updated-x-n-w}} & \mbox{ if }i'\in\mathcal{V}'_{n,w}.
\end{cases}\label{eq:update-all-x-n-w}
\end{equation}

\begin{remark}

(The case $S_{n,w}\cap\mathcal{V}=\emptyset$) A notable case is when
$S_{n,w}\cap\mathcal{V}=\emptyset$ and $|S_{n,w}\cap\mathcal{E}|=1$.
Let $(i,j)$ be the element in $S_{n,w}\cap\mathcal{E}$. Then one
can calculate from \eqref{eq:update-all-x-n-w} that $x_{i}^{n,w}=x_{j}^{n,w}=\frac{1}{2}(x_{i}^{n,w-1}+x_{j}^{n,w-1})$,
and the other $|\mathcal{V}|$ components of $x^{n,w}$ remain unchanged
from $x^{n,w-1}$. If $f_{i}(\cdot)\equiv0$ for all $i\in\mathcal{V}$
and the edges are chosen over the graph $(\mathcal{V},\mathcal{E})$,
we reduce to the case of averaged consensus studied in \cite{Boyd_distrib_averaging,Distrib_averaging_Dimakis_Kar_Moura_Rabbat_Scaglione}.

\end{remark}

\subsection{Simplification of Algorithm \ref{alg:Ext-Dyk} and further remarks}

We first remark that there is no need to track $\{z_{(i,j)}^{n,w-1}\}_{(i,j)\in\mathcal{E}}$
throughout the algorithm, and we only need to keep track of $\{z_{i}^{n,w-1}\}_{i\in\mathcal{V}}$
and $x^{n,w-1}$. We make a few more remarks about Algorithm \ref{alg:Ext-Dyk}. 

\begin{remark}(Irrelevance of $z_{(i,j)}^{n,w-1}$) \label{rem:irrelevance-of-z-i-j}A
first observation of the dual objective function is that as long as
$z_{(i,j)}\in H_{(i,j)}^{\perp}$, we have $\delta_{H_{(i,j)}}^{*}(z_{(i,j)})=\delta_{H_{(i,j)}^{\perp}}(z_{(i,j)})=0$.
Since 
\[
-\frac{1}{2}\left\Vert x_{0}-\sum_{\alpha\in\mathcal{E}\cup\mathcal{V}}z_{\alpha}^{n,w}\right\Vert ^{2}\overset{\eqref{eq:v-H-def}}{=}-\frac{1}{2}\left\Vert x_{0}-v_{H}^{n,w}-\sum_{\alpha\in\mathcal{V}}z_{\alpha}^{n,w}\right\Vert ^{2},
\]
the dual objective function \eqref{eq:Dykstra-dual-defn} thus does
not depend directly on each $\{z_{(i,j)}\}_{(i,j)\in\mathcal{E}}$,
but rather through the sum $v_{H}:=\sum_{(i,j)\in\mathcal{E}}z_{(i,j)}$
that appears in the quadratic term in \eqref{eq:Dykstra-min-subpblm}.
Next, in calculating $\frac{1}{|\mathcal{V}'_{n,w}|}\sum_{i\in\mathcal{V}'_{n,w}}\tilde{y}_{i}$
in \eqref{eq:updated-x-n-w}, we note that since $z_{(i,j)}^{n,w-1}\in H_{(i,j)}^{\perp}\subset D^{\perp}$,
$\sum_{i'\in\mathcal{V}}[z_{(i,j)}^{n,w-1}]_{i'}\overset{\eqref{eq:D-and-D-perp}}{=}0$.
Also, by Proposition \ref{prop:sparsity}, if $\alpha\in S_{n,w}\cap\mathcal{E}$,
then $[z_{\alpha}^{n,w-1}]_{i}=0$ if $i\in\mathcal{V}\backslash\mathcal{V}_{n,w}'$.
This means that 
\begin{eqnarray}
 &  & \sum_{i\in\mathcal{V}'_{n,w}}\tilde{y}_{i}\overset{\eqref{eq:tilde-y}}{=}\sum_{i\in\mathcal{V}'_{n,w}}\left[x^{n,w-1}+\sum_{\alpha\in S_{n,w}}z_{\alpha}^{n,w-1}\right]_{i}\label{eq:sum-tilde-y}\\
 & = & \sum_{i\in\mathcal{V}'_{n,w}}x_{i}^{n,w-1}+\sum_{i\in\mathcal{V}'_{n,w}}\sum_{\alpha\in S_{n,w}\cap\mathcal{V}}[z_{\alpha}^{n,w-1}]_{i}+\sum_{i\in\mathcal{V}'_{n,w}}\sum_{\alpha\in S_{n,w}\cap\mathcal{E}}[z_{\alpha}^{n,w-1}]_{i}\nonumber \\
 & = & {\normalcolor \sum_{i\in\mathcal{V}'_{n,w}}x_{i}^{n,w-1}+\sum_{i\in\mathcal{V}'_{n,w}}\sum_{\alpha\in S_{n,w}\cap\mathcal{V}}[z_{\alpha}^{n,w-1}]_{i}+\sum_{i\in\mathcal{V}}\sum_{\alpha\in S_{n,w}\cap\mathcal{E}}[z_{\alpha}^{n,w-1}]_{i}}\nonumber \\
 & \overset{\eqref{eq:D-and-D-perp}}{=} & \sum_{i\in\mathcal{V}'_{n,w}}x_{i}^{n,w-1}+\sum_{i\in\mathcal{V}'_{n,w}}\sum_{\alpha\in S_{n,w}\cap\mathcal{V}}[z_{\alpha}^{n,w-1}]_{i}.\nonumber 
\end{eqnarray}
So one only needs to keep track of $x^{n,w}$ and $\{z_{i}^{n,w}\}_{i\in\mathcal{V}}$
in Algorithm \ref{alg:Ext-Dyk}, and there is no need to keep track
of $\{z_{\alpha}^{n,w}\}_{\alpha\in\mathcal{E}}$. This justifies
why we can have the step of reassigning $z_{(i,j)}^{n,0}$ in line
4, and Algorithm \ref{alg:Ext-Dyk} could have been stated in terms
of $v_{H}$ only, and not $\{z_{(i,j)}\}_{(i,j)\in\mathcal{E}}$.
The reason why we need to introduce the variables $\{z_{(i,j)}\}_{(i,j)\in\mathcal{E}}$
is so that the analysis in \eqref{eq:biggest-formula} can be carried
through. 

\end{remark}

In view of Remark \ref{rem:irrelevance-of-z-i-j}, Algorithm \ref{alg:Ext-Dyk}
can thus be simplified to Algorithm \ref{alg:Ext-Dyk-1} without
the terms $\{z_{(i,j)}\}_{(i,j)\in\mathcal{E}}$. Furthermore, if
$f_{i}(\cdot)\equiv0$ for all $i\in V$, the variables $z_{i}^{n,w}$
would always be zero, and Algorithm \ref{alg:Ext-Dyk-1} reduces to
the well known averaged consensus problem \cite{Boyd_distrib_averaging,Distrib_averaging_Dimakis_Kar_Moura_Rabbat_Scaglione}. 

\begin{algorithm}[h]
\caption{Decentralized Dykstra's algorithm simplified}
\label{alg:Ext-Dyk-1}
Consider the problem \eqref{eq:Dyk-primal} along with the associated
dual problem \eqref{eq:dual-fn}. We only keep track of $\{z_{i}^{n,w}\}_{i\in\mathcal{V}}$
and $x^{n,w}$, and these iterates are equivalent to that of Algorithm
\ref{alg:Ext-Dyk} by Remark \ref{rem:irrelevance-of-z-i-j}.

Let $\bar{w}$ be a positive integer. Our decentralized Dykstra's
algorithm is as follows:

01$\quad$Let 

\begin{itemize}
\item $z_{i}^{1,0}\in X^{|\mathcal{V}|}$ be a starting dual vector for
$\mathbf{f}_{i}(\cdot)$ for each $i\in\mathcal{V}$ so that $[z_{i}^{1,0}]_{j}=0$
for all $j\in\mathcal{V}\backslash\{i\}$. 
\begin{itemize}
\item $v_{H}^{1,0}\in D^{\perp}$ be a starting dual vector for (1.7).
\item Let $x^{1,0}$ be $x^{1,0}=x_{0}-v_{H}^{1,0}-\sum_{i\in\mathcal{V}}z_{i}^{1,0}$.
\end{itemize}
\end{itemize}
02$\quad$For $n=1,2,\dots$

03$\quad$$\quad$Let $\mathcal{E}_{n}\subset\mathcal{E}$ be such
that $\mathcal{E}_{n}$ connects $\mathcal{V}$. 

05$\quad$$\quad$For $w=1,2,\dots,\bar{w}$

06$\quad$$\quad$$\quad$Choose a set $S_{n,w}\subset\mathcal{E}_{n}\cup\mathcal{V}$
such that $S_{n,w}\neq\emptyset$

$\phantom{\mbox{06}}$$\quad$$\quad$$\quad\qquad$and $S_{n,w}\cap\mathcal{V}\subset\mathcal{V}_{n,w}'$
for $\mathcal{V}'_{n,w}$ as defined in \eqref{eq:V-prime}.

07$\quad$$\quad$$\quad$Define $x^{n,w}\in X^{|\mathcal{V}|}$ by
\begin{equation}
x_{i'}^{n,w}\overset{\eqref{eq:update-all-x-n-w}}{=}\begin{cases}
x_{i'}^{n,w-1} & \mbox{ if }i'\notin\mathcal{V}'_{n,w}\\
\underset{x\in X}{\arg\min}\sum_{i\in S_{n,w}\cap\mathcal{V}}f_{i}(x)+\frac{|\mathcal{V}'_{n,w}|}{2}\left\Vert x-\frac{1}{|\mathcal{V}'_{n,w}|}\sum_{i\in\mathcal{V}'_{n,w}}\tilde{y}_{i}\right\Vert ^{2} & \mbox{ if }i'\in\mathcal{V}'_{n,w}.
\end{cases}\label{eq:update-all-x-n-w-1}
\end{equation}
where $\sum_{i\in\mathcal{V}'_{n,w}}\tilde{y}_{i}$ has the form in
\eqref{eq:sum-tilde-y}, which does not depend on $\{z_{(i,j)}^{n,w-1}\}_{(i,j)\in\mathcal{E}}$.
Let $\{z_{i}^{n,w}\}_{i\in S_{n,w}\cap\mathcal{V}}$ be such that
$[z_{i}^{n,w}]_{i}$ is the subgradient of $f_{i}(\cdot)$ at $x_{i'}^{n,w}$
that certifies the optimality in \eqref{eq:update-all-x-n-w-1}.

08$\quad$$\quad$$\quad$Set $z_{\alpha}^{n,w}:=z_{\alpha}^{n,w-1}$
for all $\alpha\notin S_{n,w}\cap\mathcal{V}$.

09$\quad$$\quad$End For 

10$\quad$$\quad$Let $z_{i}^{n+1,0}=z_{i}^{n,\bar{w}}$ for all $i\in\mathcal{V}$,
and $x^{n+1,0}=x^{n,\bar{w}}$.

11$\quad$End For 
\end{algorithm}
\begin{remark}

\label{rem:distrib-comp} (Distributed asynchronous computation) Proposition
\ref{prop:sparsity} shows that the storage requirement for each vertex
and edge is small. Suppose $S_{n,w}$ and $S_{n,w+1}$ are such that
$S_{n,w}\cap\mathcal{V}\subset\mathcal{V}'_{n,w}$, $S_{n,w+1}\cap\mathcal{V}\subset\mathcal{V}_{n,w+1}'$
and $\mathcal{V}_{n,w}'\cap\mathcal{V}_{n,w+1}'=\emptyset$. Then
the computations in for the iterations $(n,w)$ and $(n,w+1)$ can
be conducted in parallel. This is because calculations for $S_{n,w}$
in \eqref{eq:Dykstra-min-subpblm} only uses and affects the coordinates
of $\{z_{\alpha}\}_{\alpha\in\mathcal{E}\cup\mathcal{V}}$ indexed
by $\mathcal{V}'_{n,w}\subset\mathcal{V}$ and the similar thing goes
for $S_{n,w+1}$. This idea can be naturally extended to the case
of $S_{n,w},S_{n,w+1},\dots,S_{n,w+j}$ for any $j\geq1$ to allows
for distributed asynchronous computation.

\end{remark}

\begin{remark}

(Scalability) Algorithm \ref{alg:Ext-Dyk} allows for the size of
the sets $S_{n,w}$ to be arbitrarily large so that there would be
a greedier increase in the dual objective value. One would then expect
faster convergence with larger sizes of $S_{n,w}$. Even though for
this paper, we only cover the case where $|S_{n,w}\cap\mathcal{V}|\leq1$,
the case where $|S_{n,w}\cap\mathcal{V}|>1$ can be analyzed using
the techniques in \cite{Pang_Dyk_spl}, where we split vertices in
$\mathcal{V}$ according to whether $\dom(f_{i})=X$, $f_{i}(\cdot)$
is an indicator function of a closed convex set, or $f_{i}(\cdot)$
is a general closed convex function.

\end{remark}

\begin{remark}

\label{rem:time-vary} (Time-varying graphs) Note that in line 5 of
Algorithm \ref{alg:Ext-Dyk}, we only need to choose $\mathcal{E}_{n}\subset\mathcal{E}$
so that $\mathcal{E}_{n}$ connects $\mathcal{V}$. As long as $\mathcal{E}_{n}=[\cup_{w=1}^{\bar{w}}S_{n,w}]\cap\mathcal{E}$,
the convergence result in Theorem \ref{thm:convergence} holds. So
as long as enough edges are chosen in each cycle to connect the graph,
Algorithm \ref{alg:Ext-Dyk} would converge. In \cite{Nedich_Olshevsky},
they used the term uniformly strongly connectedness or $B$-strongly
connectedness for time-varying directed graphs. Our assumption is
equivalent to how $B$-connectedness would have been defined for undirected
graphs. 

\end{remark}

\subsection{Convergence of Algorithm \ref{alg:Ext-Dyk}}

We state some notation necessary for further discussions. For any
$\alpha\in\mathcal{E}\cup\mathcal{V}$ and $n\in\{1,2,\dots\}$, let
$p(n,\alpha)$ be 
\[
p(n,\alpha)=\max\{m:m\leq\bar{w},\alpha\in S_{n,m}\}.
\]
In other words, $p(n,\alpha)$ is the index $m$ such that $\alpha\in S_{n,m}$
but $\alpha\notin S_{n,k}$ for all $k\in\{m+1,\dots,\bar{w}\}$.
It follows from line 8 in Algorithm \ref{alg:Ext-Dyk} that 
\begin{equation}
z_{\alpha}^{n,p(n,\alpha)}=z_{\alpha}^{n,p(n,\alpha)+1}=\cdots=z_{\alpha}^{n,\bar{w}}.\label{eq:stagnant-indices}
\end{equation}
Moreover, $(i,j)\notin\mathcal{E}_{n}$ implies $(i,j)\notin S_{n,w}$
for all $w\in\{1,\dots,\bar{w}\}$, so 
\begin{equation}
0\overset{\scriptsize\eqref{eq:reset-z-i-j-1}}{=}z_{(i,j)}^{n,0}=z_{(i,j)}^{n,1}=\cdots=z_{(i,j)}^{n,\bar{w}}\mbox{ for all }(i,j)\notin\mathcal{E}_{n}.\label{eq:zero-indices}
\end{equation}

We have the following theorem on the convergence of Algorithm \ref{alg:Ext-Dyk}.
\begin{theorem}
\label{thm:convergence} (Convergence to primal minimizer) Consider
Algorithm \ref{alg:Ext-Dyk}. Assume that for all $n\geq1$, $\mathcal{E}_{n}=[\cup_{w=1}^{\bar{w}}S_{n,w}]\cap\mathcal{E}$,
and $[\cup_{w=1}^{\bar{w}}S_{n,w}]\supset\mathcal{V}$. Assume that
there are constants $A$ and $B$ such that 
\begin{equation}
\sum_{\alpha\in\mathcal{E}\cup\mathcal{V}}\|z_{\alpha}^{n,\bar{w}}\|\leq A\sqrt{n}+B\mbox{ for all }n\geq0.\label{eq:sqrt-growth-sum-z}
\end{equation}

For the sequence $\{z_{\alpha}^{n,w}\}_{{1\leq n<\infty\atop 0\leq w\leq\bar{w}}}\subset X^{|\mathcal{V}|}$
for each $\alpha\in\mathcal{E}\cup\mathcal{V}$ generated by Algorithm
\ref{alg:Ext-Dyk} and the sequences $\{v_{H}^{n,w}\}_{{1\leq n<\infty\atop 0\leq w\leq\bar{w}}}\subset X^{|\mathcal{V}|}$
and $\{v_{A}^{n,w}\}_{{1\leq n<\infty\atop 0\leq w\leq\bar{w}}}\subset X^{|\mathcal{V}|}$
thus derived, we have:

\begin{enumerate}
\item [(i)]The sum $\sum_{n=1}^{\infty}\sum_{w=1}^{\bar{w}}\|v_{A}^{n,w}-v_{A}^{n,w-1}\|^{2}$
is finite and $\{F(\{z_{\alpha}^{n,\bar{w}}\}_{\alpha\in\mathcal{E}\cup\mathcal{V}})\}_{n=1}^{\infty}$
is nondecreasing.
\item [(ii)]There is a constant $C$ such that $\|v_{A}^{n,w}\|^{2}\leq C$
for all $n\in\mathbb{N}$ and $w\in\{1,\dots,\bar{w}\}$. 
\item [(iii)]There exists a subsequence $\{v_{A}^{n_{k},\bar{w}}\}_{k=1}^{\infty}$
of $\{v_{A}^{n,\bar{w}}\}_{n=1}^{\infty}$ which converges to some
$v_{A}^{*}\in X^{|\mathcal{V}|}$ and that 
\[
\lim_{k\to\infty}\langle v_{A}^{n_{k},\bar{w}}-v_{A}^{n_{k},p(n_{k},\alpha)},z_{\alpha}^{n_{k},\bar{w}}\rangle=0\mbox{ for all }\alpha\in\mathcal{E}\cup\mathcal{V}.
\]
\item [(iv)]For the $v_{A}^{*}$ in (iii), $x_{0}-v_{A}^{*}$ is the minimizer
of the primal problem (P) and we have $\lim_{k\to\infty}F(\{z_{\alpha}^{n_{k},w}\}_{\alpha\in\mathcal{E}\cup\mathcal{V}})=\frac{1}{2}\|v_{A}^{*}\|^{2}+h(x_{0}-v_{A}^{*})$,
where $h(\cdot)=\sum_{\alpha\in\mathcal{E}\cup\mathcal{V}}h_{\alpha}(\cdot)$. 
\end{enumerate}
The properties (i) to (iv) in turn imply that $\lim_{n\to\infty}x^{n,\bar{w}}$
exists and equals $x_{0}-v_{A}^{*}$, which is the primal minimizer
of \eqref{eq:Dyk-primal}.
\end{theorem}

\begin{proof}
We first show that (i) to (iv) implies the final assertion. For all
$n\in\mathbb{N}$ we have, from weak duality, 
\begin{equation}
\begin{array}{c}
F(\{z_{\alpha}^{n,\bar{w}}\}_{\alpha\in\mathcal{E}\cup\mathcal{V}})\leq\frac{1}{2}\|x_{0}-(x_{0}-v_{A}^{*})\|^{2}+\underset{\alpha\in\mathcal{E}\cup\mathcal{V}}{\overset{}{\sum}}h_{\alpha}(x_{0}-v_{A}^{*}).\end{array}\label{eq:weak-duality}
\end{equation}
Since the values $\{F(\{z_{\alpha}^{n,\bar{w}}\}_{\alpha\in\mathcal{E}\cup\mathcal{V}})\}_{n=1}^{\infty}$
are nondecreasing in $n$, we make use of (iv) to get 
\[
\begin{array}{c}
\underset{n\to\infty}{\lim}F(\{z_{\alpha}^{n,\bar{w}}\}_{\alpha\in\mathcal{E}\cup\mathcal{V}})=\frac{1}{2}\|x_{0}-(x_{0}-v_{A}^{*})\|^{2}+\underset{\alpha\in\mathcal{E}\cup\mathcal{V}}{\overset{}{\sum}}h_{\alpha}(x_{0}-v_{A}^{*}),\end{array}
\]
Hence $x_{0}-v_{A}^{*}=\arg\min_{x}h(x)+\frac{1}{2}\|x-x_{0}\|^{2}$,
and (substituting $x=x_{0}-v_{A}^{*}$ in \eqref{eq:From-8}) 
\begin{eqnarray*}
 &  & \begin{array}{c}
\frac{1}{2}\|x_{0}-(x_{0}-v_{A}^{*})\|^{2}+h(x_{0}-v_{A}^{*})-F(\{z_{\alpha}^{n,\bar{w}}\}_{\alpha\in\mathcal{E}\cup\mathcal{V}})\end{array}\\
 & \overset{\eqref{eq:From-8},\eqref{eq:v-H-def},\eqref{eq:from-10}}{\geq} & \begin{array}{c}
\frac{1}{2}\|x_{0}-(x_{0}-v_{A}^{*})-v_{A}^{n,\bar{w}}\|^{2}\end{array}\\
 & \overset{\eqref{eq:x-from-v-A}}{=} & \begin{array}{c}
\frac{1}{2}\|x^{n,\bar{w}}-(x_{0}-v_{A}^{*})\|^{2}.\end{array}
\end{eqnarray*}
Hence $\lim_{n\to\infty}x^{n,\bar{w}}$ is the minimizer in (P). 

It remains to prove assertions (i) to (iv).

\textbf{Proof of (i):} From the fact that $\{z_{\alpha}^{n,w}\}_{\alpha\in S_{n,w}}$
minimize \eqref{eq:Dykstra-min-subpblm} (which includes the quadratic
regularizer) we have 
\begin{eqnarray}
F(\{z_{\alpha}^{n,w-1}\}_{\alpha\in\mathcal{E}\cup\mathcal{V}}) & \overset{\eqref{eq:Dykstra-min-subpblm}}{\leq} & \begin{array}{c}
F(\{z_{\alpha}^{n,w}\}_{\alpha\in\mathcal{E}\cup\mathcal{V}})-\frac{1}{2}\|v_{A}^{n,w}-v_{A}^{n,w-1}\|^{2}.\end{array}\label{eq:SHQP-decrease}
\end{eqnarray}
(The last term in \eqref{eq:SHQP-decrease} arises from the quadratic
term in \eqref{eq:Dykstra-min-subpblm}.) By line 10 of Algorithm
\ref{alg:Ext-Dyk}, $z_{i}^{n+1,0}=z_{i}^{n,\bar{w}}$ for all $i\in\mathcal{V}$
and $v_{H}^{n+1,0}=v_{H}^{n,\bar{w}}$ (even though the decompositions
\eqref{eq:reset-z-i-j-4} of $v_{H}^{n+1,0}$ and $v_{H}^{n,\bar{w}}$
may be different). Combining \eqref{eq:SHQP-decrease} over all $m\in\{1,\dots,n\}$
and $w\in\{1,\dots,\bar{w}\}$, we have 
\[
\begin{array}{c}
F(\{z_{\alpha}^{1,0}\}_{\alpha\in\mathcal{E}\cup\mathcal{V}})+\underset{m=1}{\overset{n}{\sum}}\underset{w=1}{\overset{\bar{w}}{\sum}}\|v_{A}^{m,w}-v_{A}^{m,w-1}\|^{2}\overset{\eqref{eq:SHQP-decrease}}{\leq}F(\{z_{\alpha}^{n,\bar{w}}\}_{\alpha\in\mathcal{E}\cup\mathcal{V}}).\end{array}
\]
Next, $F(\{z_{\alpha}^{n,\bar{w}}\}_{\alpha\in\mathcal{E}\cup\mathcal{V}})$
is bounded from above by weak duality. The proof of the claim is complete.

\textbf{Proof of (ii):} Substituting $\{z_{\alpha}\}_{\alpha\in\mathcal{E}\cup\mathcal{V}}$
in \eqref{eq:From-8} to be $\{z_{\alpha}^{n,w}\}_{\alpha\in\mathcal{E\cup}\mathcal{V}}$
and $x$ to be the primal minimizer $x^{*}$, we have 
\begin{eqnarray*}
 &  & \begin{array}{c}
\frac{1}{2}\|x_{0}-x^{*}\|^{2}+\underset{\alpha\in\mathcal{E}\cup\mathcal{V}}{\overset{}{\sum}}h_{\alpha}(x^{*})-F(\{z_{\alpha}^{1,0}\}_{\alpha\in\mathcal{E}\cup\mathcal{V}})\end{array}\\
 & \overset{\scriptsize\mbox{part (i)}}{\geq} & \begin{array}{c}
\frac{1}{2}\|x_{0}-x^{*}\|^{2}+\underset{\alpha\in\mathcal{E}\cup\mathcal{V}}{\overset{}{\sum}}h_{\alpha}(x^{*})-F(\{z_{\alpha}^{n,w}\}_{\alpha\in\mathcal{E}\cup\mathcal{V}})\end{array}\\
 & \overset{\eqref{eq:From-8}}{\geq} & \begin{array}{c}
\frac{1}{2}\left\Vert x_{0}-x^{*}-\underset{\alpha\in\mathcal{E}\cup\mathcal{V}}{\overset{}{\sum}}z_{\alpha}^{n,w}\right\Vert ^{2}\overset{\eqref{eq:from-10}}{=}\frac{1}{2}\|x_{0}-x^{*}-v_{A}^{n,w}\|^{2}.\end{array}
\end{eqnarray*}
The conclusion is immediate.

\textbf{Proof of (iii): }We first make use of the technique in \cite[Lemma 29.1]{BauschkeCombettes11}
(which in turn is largely attributed to \cite{BD86}) to show that
\begin{equation}
\begin{array}{c}
\underset{n\to\infty}{\liminf}\left[\left(\underset{w=1}{\overset{\bar{w}}{\sum}}\|v_{A}^{n,w}-v_{A}^{n,w-1}\|\right)\sqrt{n}\right]=0.\end{array}\label{eq:root-n-dec}
\end{equation}
Seeking a contradiction, suppose instead that there is an $\epsilon>0$
and $\bar{n}>0$ such that if $n>\bar{n}$, then $\left(\sum_{w=1}^{\bar{w}}\|v_{A}^{n,w}-v_{A}^{n,w-1}\|\right)\sqrt{n}>\epsilon$.
By the Cauchy Schwarz inequality, we have $\begin{array}{c}
\frac{\epsilon^{2}}{n}<\left(\underset{w=1}{\overset{\bar{w}}{\sum}}\|v_{A}^{n,w}-v_{A}^{n,w-1}\|\right)^{2}\leq\bar{w}\underset{w=1}{\overset{\bar{w}}{\sum}}\|v_{A}^{n,w}-v_{A}^{n,w-1}\|^{2}.\end{array}$ This contradicts the earlier claim in (i) that $\sum_{n=1}^{\infty}\sum_{w=1}^{\bar{w}}\|v_{A}^{n,w}-v_{A}^{n,w-1}\|^{2}$
is finite. 

Through \eqref{eq:root-n-dec}, we find a sequence $\{n_{k}\}_{k=1}^{\infty}$
such that 
\begin{equation}
\lim_{k\to\infty}\left[\left(\sum_{w=1}^{\bar{w}}\|v_{A}^{n_{k},w}-v_{A}^{n_{k},w-1}\|\right)\sqrt{n_{k}}\right]=0.\label{eq:subseq-sqrt-limit}
\end{equation}
Recalling the assumption \eqref{eq:sqrt-growth-sum-z}, we get 
\begin{equation}
\begin{array}{c}
\underset{k\to\infty}{\lim}\left[\left(\underset{w=1}{\overset{\bar{w}}{\sum}}\|v_{A}^{n_{k},w}-v_{A}^{n_{k},w-1}\|\right)\|z_{\alpha}^{n_{k},\bar{w}}\|\right]\overset{\eqref{eq:sqrt-growth-sum-z},\eqref{eq:subseq-sqrt-limit}}{=}0\mbox{ for all }\alpha\in\mathcal{E}\cup\mathcal{V}.\end{array}\label{eq:lim-sum-norm-z}
\end{equation}
Moreover, 
\begin{eqnarray}
|\langle v_{A}^{n_{k},\bar{w}}-v_{A}^{n_{k},p(n_{k},\alpha)},z_{\alpha}^{n_{k},\bar{w}}\rangle| & \leq & \begin{array}{c}
\|v_{A}^{n_{k},\bar{w}}-v_{A}^{n_{k},p(n_{k},\alpha)}\|\|z_{\alpha}^{n_{k},\bar{w}}\|\end{array}\label{eq:inn-pdt-sum-norm}\\
 & \leq & \begin{array}{c}
\left(\underset{w=1}{\overset{\bar{w}}{\sum}}\|v_{A}^{n_{k},w}-v_{A}^{n_{k},w-1}\|\right)\|z_{\alpha}^{n_{k},\bar{w}}\|.\end{array}\nonumber 
\end{eqnarray}
By (ii) and the finite dimensionality of $X$, there exists a further
subsequence of $\{v_{A}^{n_{k},\bar{w}}\}_{k=1}^{\infty}$ which converges
to some $v_{A}^{*}\in X$. Combining \eqref{eq:lim-sum-norm-z} and
\eqref{eq:inn-pdt-sum-norm} gives (iii).

\textbf{Proof of (iv):} From earlier results, we obtain 
\begin{eqnarray}
 &  & \begin{array}{c}
-\underset{\alpha\in\mathcal{E}\cup\mathcal{V}}{\overset{}{\sum}}h_{\alpha}(x_{0}-v_{A}^{*})\end{array}\label{eq:biggest-formula}\\
 & \overset{\eqref{eq:From-8}}{\leq} & \begin{array}{c}
\frac{1}{2}\|x_{0}-(x_{0}-v_{A}^{*})\|^{2}-F(\{z_{\alpha}^{n_{k},\bar{w}}\}_{\alpha\in\mathcal{E}\cup\mathcal{V}})\end{array}\nonumber \\
 & \overset{\eqref{eq:Dykstra-dual-defn},\eqref{eq:stagnant-indices}}{=} & \begin{array}{c}
\frac{1}{2}\|v_{A}^{*}\|^{2}+\underset{\alpha\in\mathcal{E}_{n_{k}}\cup\mathcal{V}}{\overset{}{\sum}}h_{\alpha}^{*}(z_{\alpha}^{n_{k},p(n_{k},\alpha)})\end{array}\nonumber \\
 &  & \begin{array}{c}
+\underset{(i,j)\notin\mathcal{E}_{n_{k}}}{\overset{}{\sum}}h_{(i,j)}^{*}(z_{(i,j)}^{n_{k},\bar{w}})-\langle x_{0},v_{A}^{n_{k},\bar{w}}\rangle+\frac{1}{2}\|v_{A}^{n_{k},\bar{w}}\|^{2}\end{array}\nonumber \\
 & \overset{\scriptsize\mbox{Claim \ref{claim:Fenchel-duality}(c)},\alpha\in S_{n,p(n,\alpha)},\eqref{eq:zero-indices}}{=} & \begin{array}{c}
\frac{1}{2}\|v_{A}^{*}\|^{2}+\underset{\alpha\in\mathcal{E}_{n_{k}}\cup\mathcal{V}}{\overset{}{\sum}}\langle x_{0}-v_{A}^{n_{k},p(n_{k},\alpha)},z_{\alpha}^{n_{k},p(n_{k},\alpha)}\rangle\end{array}\nonumber \\
 &  & \begin{array}{c}
-\underset{\alpha\in\mathcal{E}_{n_{k}}\cup\mathcal{V}}{\overset{}{\sum}}h_{\alpha}(x_{0}-v_{A}^{n_{k},p(n_{k},\alpha)})-\langle x_{0},v_{A}^{n_{k},\bar{w}}\rangle+\frac{1}{2}\|v_{A}^{n_{k},\bar{w}}\|^{2}\end{array}\nonumber \\
 & \overset{\eqref{eq:stagnant-indices}}{=} & \begin{array}{c}
\frac{1}{2}\|v_{A}^{*}\|^{2}-\underset{\alpha\in\mathcal{E}_{n_{k}}\cup\mathcal{V}}{\overset{}{\sum}}\langle v_{A}^{n_{k},p(n_{k},\alpha)}-v_{A}^{n_{k},\bar{w}},z_{\alpha}^{n_{k},\bar{w}}\rangle\end{array}\nonumber \\
 &  & \begin{array}{c}
-\underset{\alpha\in\mathcal{E}_{n_{k}}\cup\mathcal{V}}{\overset{}{\sum}}h_{\alpha}(x_{0}-v_{A}^{n_{k},p(n_{k},\alpha)})-\langle x_{0},v_{A}^{n_{k},\bar{w}}\rangle\end{array}\nonumber \\
 &  & \begin{array}{c}
+\left\langle x_{0}-v_{A}^{n_{k},\bar{w}},\underset{\alpha\in\mathcal{E}_{n_{k}}\cup\mathcal{V}}{\overset{}{\sum}}z_{\alpha}^{n_{k},p(n_{k},\alpha)}\right\rangle +\frac{1}{2}\|v_{A}^{n_{k},\bar{w}}\|^{2}\end{array}\nonumber \\
 & \overset{\eqref{eq:from-10},\eqref{eq:zero-indices}}{=} & \begin{array}{c}
\frac{1}{2}\|v_{A}^{*}\|^{2}-\frac{1}{2}\|v_{A}^{n_{k},\bar{w}}\|^{2}-\!\!\!\underset{\alpha\in\mathcal{E}_{n_{k}}\cup\mathcal{V}}{\overset{}{\sum}}\langle v_{A}^{n_{k},p(n_{k},\alpha)}-v_{A}^{n_{k},\bar{w}},z_{\alpha}^{n_{k},\bar{w}}\rangle\end{array}\nonumber \\
 &  & \begin{array}{c}
-\underset{\alpha\in\mathcal{E}_{n_{k}}\cup\mathcal{V}}{\overset{}{\sum}}h_{\alpha}(x_{0}-v_{A}^{n_{k},p(n_{k},\alpha)})\end{array}\nonumber 
\end{eqnarray}

Since $\lim_{k\to\infty}v_{A}^{n_{k},\bar{w}}=v_{A}^{*}$, we have
$\lim_{k\to\infty}\frac{1}{2}\|v_{A}^{*}\|^{2}-\frac{1}{2}\|v_{A}^{n_{k},\bar{w}}\|^{2}=0$.
The third term in the last group of formulas (i.e., the sum involving
the inner products) converges to 0 by (iii).

Next, recall that if $(i,j)\in\mathcal{E}_{n}$, by \eqref{eq:primal-subpblm},
we have $h_{(i,j)}(x_{0}-v_{A}^{n,p(n,(i,j))})=0$, which gives $x_{0}-v_{A}^{n,p(n,(i,j))}\overset{\eqref{eq:common-primal},\eqref{eq:H-i-j-subspace}}{\in}H_{(i,j)}$.
There is a constant $\kappa_{\mathcal{E}_{n_{k}}}>0$ such that 
\begin{eqnarray}
 &  & d(x_{0}-v_{A}^{n_{k},\bar{w}},\cap_{(i,j)\in\mathcal{E}}H_{(i,j)})\label{eq:reg-argument}\\
 & \overset{\scriptsize{\mathcal{E}_{n_{k}}\mbox{ connects }\mathcal{V},\mbox{ Prop \ref{prop:E-connects-V}(1)}}}{=} & d(x_{0}-v_{A}^{n_{k},\bar{w}},\cap_{(i,j)\in\mathcal{E}_{n_{k}}}H_{(i,j)})\nonumber \\
 & \leq & \kappa_{\mathcal{E}_{n_{k}}}\max_{(i,j)\in\mathcal{E}_{n_{k}}}d(x_{0}-v_{A}^{n_{k},\bar{w}},H_{(i,j)})\nonumber \\
 & \overset{x_{0}-v_{A}^{n_{k},p(n_{k},(i,j))}\in H_{(i,j)}}{\leq} & \kappa_{\mathcal{E}_{n_{k}}}\max_{(i,j)\in\mathcal{E}_{n_{k}}}\|v_{A}^{n_{k},\bar{w}}-v_{A}^{n_{k},p(n_{k},(i,j))}\|.\nonumber 
\end{eqnarray}
Let $\kappa:=\max\{\kappa_{\mathcal{E}'}:\mathcal{E}'\mbox{ connects }\mathcal{V}\}$.
We have $\kappa_{\mathcal{E}_{n_{k}}}\leq\kappa$. Taking limits of
\eqref{eq:reg-argument}, the RHS converges to zero by (i), so $d(x_{0}-v_{A}^{*},\cap_{(i,j)\in\mathcal{E}}H_{(i,j)})=0$,
or $x_{0}-v_{A}^{*}\in\cap_{(i,j)\in\mathcal{E}}H_{(i,j)}$. So $\sum_{(i,j)\in\mathcal{E}}h_{(i,j)}(x_{0}-v_{A}^{*})=0$.
Together with the fact that $x_{0}-v_{A}^{n_{k},p(n_{k},(i,j))}\in H_{(i,j)}$,
we have 
\begin{equation}
\sum_{(i,j)\in\mathcal{E}_{n_{k}}}h_{(i,j)}(x_{0}-v_{A}^{n_{k},p(n_{k},(i,j))})=0=\sum_{(i,j)\in\mathcal{E}}h_{(i,j)}(x_{0}-v_{A}^{*}).\label{eq:all-indicator-edges-zero}
\end{equation}

Lastly, by the lower semicontinuity of $h_{i}(\cdot)$, we have 
\begin{equation}
-\lim_{k\to\infty}\sum_{i\in\mathcal{V}}h_{i}(x_{0}-v_{A}^{n_{k},p(n_{k},i)})\leq-\sum_{i\in\mathcal{V}}h_{i}(x_{0}-v_{A}^{*}).\label{eq:lsc-argument}
\end{equation}
As mentioned after \eqref{eq:biggest-formula}, taking the limits
as $k\to\infty$ would result in the first three terms of the last
formula in \eqref{eq:biggest-formula} to be zero. Hence 
\begin{eqnarray*}
-\sum_{\alpha\in\mathcal{E}\cup\mathcal{V}}h_{\alpha}(x_{0}-v_{A}^{*}) & \overset{\eqref{eq:biggest-formula}}{\leq} & \lim_{k\to\infty}-\sum_{\alpha\in\mathcal{E}_{n_{k}}\cup\mathcal{V}}h_{\alpha}(x_{0}-v_{A}^{n_{k},p(n_{k},\alpha)})\\
 & \overset{\eqref{eq:all-indicator-edges-zero},\eqref{eq:lsc-argument}}{\leq} & -\sum_{\alpha\in\mathcal{E}\cup\mathcal{V}}h_{\alpha}(x_{0}-v_{A}^{*}).
\end{eqnarray*}
So \eqref{eq:biggest-formula} becomes an equation in the limit. The
first two lines of \eqref{eq:biggest-formula} then gives
\[
\lim_{k\to\infty}F(\{z_{\alpha}^{n_{k},\bar{w}}\}_{\alpha\in\mathcal{E}\cup\mathcal{V}})=\frac{1}{2}\|v_{A}^{*}\|^{2}+\sum_{i\in\mathcal{V}}h_{i}(x_{0}-v_{A}^{*}),
\]
which shows that $x_{0}-v_{A}^{*}$ is the primal minimizer. 
\end{proof}
A last detail that we need to resolve is to show that \eqref{eq:sqrt-growth-sum-z}
holds for the choice of $S_{n,w}$ in Algorithm \ref{alg:Ext-Dyk}. 
\begin{proposition}
(Growth of $\sum_{\alpha\in\mathcal{E}\cup\mathcal{V}}\|z_{\alpha}^{n,w}\|$)
If $S_{n,w}$ are such that $|S_{n,w}\cap\mathcal{V}|\leq1$ for all
$n\in\mathbb{N}$ and $w\in\{1,\dots,\bar{w}\}$ like in Algorithm
\ref{alg:Ext-Dyk}, then \eqref{eq:sqrt-growth-sum-z} holds. 
\end{proposition}

\begin{proof}
We either have $S_{n,w}\cap\mathcal{V}=\emptyset$ or $|S_{n,w}\cap\mathcal{V}|=1$.
In the second case, let $i^{*}$ be the index such that $i^{*}\in S_{n,w}\cap\mathcal{V}$.
Otherwise, in the first case, we let $i^{*}$ be any index in $\mathcal{V}$.
We have 
\begin{eqnarray}
\sum_{i\in\mathcal{V}}[v_{A}^{n,w}-v_{A}^{n,w-1}]_{i} & \overset{\scriptsize{\mbox{line 8},\eqref{eq_m:all_acronyms}}}{=} & \sum_{i\in\mathcal{V}}\sum_{\alpha\in S_{n,w}}[z_{\alpha}^{n,w}-z_{\alpha}^{n,w-1}]_{i}\label{eq:for-norm-rate}\\
 & \overset{z_{(i,j)}\in D^{\perp},\eqref{eq:D-and-D-perp}}{=} & \sum_{i\in\mathcal{V}}[z_{i^{*}}^{n,w}-z_{i^{*}}^{n,w-1}]_{i}\nonumber \\
 & \overset{\scriptsize{\mbox{Prop \ref{prop:sparsity}}}}{=} & [z_{i^{*}}^{n,w}-z_{i^{*}}^{n,w-1}]_{i^{*}}.\nonumber 
\end{eqnarray}
Recall that the norm $\|\cdot\|$ always refers to the $2$-norm unless
stated otherwise. By the equivalence of norms in finite dimensions,
we can find a constant $c_{1}$ such that 
\begin{eqnarray}
\|v_{A}^{n,w}-v_{A}^{n,w-1}\| & \geq & c_{1}\sum_{i\in\mathcal{V}}\|[v_{A}^{n,w}-v_{A}^{n,w-1}]_{i}\|\label{eq:bdd-z-i}\\
 & \geq & c_{1}\left\Vert \sum_{i\in\mathcal{V}}[v_{A}^{n,w}-v_{A}^{n,w-1}]_{i}\right\Vert \nonumber \\
 & \overset{\eqref{eq:for-norm-rate}}{=} & c_{1}\|z_{i^{*}}^{n,w}-z_{i^{*}}^{n,w-1}\|\overset{\eqref{eq:Dykstra-min-subpblm}}{=}c_{1}\sum_{i\in\mathcal{V}}\|z_{i}^{n,w}-z_{i}^{n,w-1}\|.\nonumber 
\end{eqnarray}

Next, $v_{H}^{n,w}-v_{H}^{n,w-1}\overset{\eqref{eq:from-10}}{=}v_{A}^{n,w}-v_{A}^{n,w-1}-(z_{i^{*}}^{n,w}-z_{i^{*}}^{n,w-1})$,
so 
\begin{eqnarray}
\|v_{H}^{n,w}-v_{H}^{n,w-1}\| & \leq & \|v_{A}^{n,w}-v_{A}^{n,w-1}\|+\|z_{i^{*}}^{n,w}-z_{i^{*}}^{n,w-1}\|\label{eq:bdd-v-H}\\
 & \overset{\eqref{eq:bdd-z-i}}{\leq} & \left(1+\frac{1}{c_{1}}\right)\|v_{A}^{n,w}-v_{A}^{n,w-1}\|.\nonumber 
\end{eqnarray}
We can choose $\{z_{(i,j)}^{n,w}\}_{(i,j)\in\mathcal{E}}$ such that
\begin{equation}
\sum_{(i,j)\in S_{n,w}\cap\mathcal{E}}[z_{(i,j)}^{n,w}-z_{(i,j)}^{n,w-1}]\overset{\scriptsize{\mbox{line 8}}}{=}\sum_{(i,j)\in\mathcal{E}}[z_{(i,j)}^{n,w}-z_{(i,j)}^{n,w-1}]\overset{\eqref{eq:v-H-def}}{=}v_{H}^{n,w}-v_{H}^{n,w-1}.\label{eq:decomp-v-H}
\end{equation}
Without loss of generality, we can assume that $S_{n,w}\cap\mathcal{E}$
contains edges that do not form a cycle. This also means that for
a $v_{H}^{n,w}-v_{H}^{n,w-1}$, each $z_{(i,j)}^{n,w}-z_{(i,j)}^{n,w-1}$
can be determined uniquely with a linear map from the relation \eqref{eq:decomp-v-H}.
Therefore there is a constant $\kappa_{(i,j),S_{n,w}\cap\mathcal{E}}>0$
such that 
\begin{equation}
\|z_{(i,j)}^{n,w}-z_{(i,j)}^{n,w-1}\|\leq\kappa_{(i,j),S_{n,w}\cap\mathcal{E}}\|v_{H}^{n,w}-v_{H}^{n,w-1}\|.\label{eq:basic-bdd-z-i-j}
\end{equation}
Thus there is a constant $\kappa>0$ such that 
\begin{equation}
\sum_{(i,j)\in\mathcal{E}}\|z_{(i,j)}^{n,w}-z_{(i,j)}^{n,w-1}\|\overset{\eqref{eq:Dykstra-min-subpblm}}{=}\sum_{(i,j)\in S_{n,w}\cap\mathcal{E}}\|z_{(i,j)}^{n,w}-z_{(i,j)}^{n,w-1}\|\overset{\eqref{eq:basic-bdd-z-i-j}}{\leq}\kappa\|v_{H}^{n,w}-v_{H}^{n,w-1}\|.\label{eq:bdd-z-i-j}
\end{equation}
Combining \eqref{eq:bdd-z-i}, \eqref{eq:bdd-v-H} and \eqref{eq:bdd-z-i-j}
together shows that there is a constant $C_{2}>1$ such that 
\begin{equation}
\|v_{H}^{n,w}-v_{H}^{n,w-1}\|+\sum_{(i,j)\in\mathcal{E}}\|z_{(i,j)}^{n,w}-z_{(i,j)}^{n,w-1}\|+\sum_{i\in\mathcal{V}}\|z_{i}^{n,w}-z_{i}^{n,w-1}\|\leq C_{2}\|v_{A}^{n,w}-v_{A}^{n,w-1}\|.\label{eq:all-3-bdd}
\end{equation}
Since $\{z_{\alpha}^{n,0}\}_{\alpha\in\mathcal{E}}$ was chosen to
satisfy \eqref{eq_m:resetted-z-i-j}, there is some $M>1$ such that
\begin{equation}
\sum_{\alpha\in\mathcal{E}}\|z_{\alpha}^{n,0}\|\overset{\eqref{eq:reset-z-i-j-3}}{\leq}M\|v_{H}^{n,0}\|\overset{\eqref{eq:reset-z-i-j-4}}{\leq}M\left(\|v_{H}^{1,0}\|+\sum_{m=1}^{n-1}\sum_{w=1}^{\bar{w}}\|v_{H}^{m,w}-v_{H}^{m,w-1}\|\right)\label{eq:z-bdd-for-E}
\end{equation}
Now for any $n\geq1$, we have 
\begin{eqnarray}
\sum_{\alpha\in\mathcal{E}\cup\mathcal{V}}\|z_{\alpha}^{n,\bar{w}}\|\!\!\!\! & \leq & \sum_{\alpha\in\mathcal{E}}\|z_{\alpha}^{n,0}\|+\sum_{w=1}^{\bar{w}}\sum_{\alpha\in\mathcal{E}}\|z_{\alpha}^{n,w}-z_{\alpha}^{n,w-1}\|\label{eq:2nd-big-ineq}\\
 &  & +\sum_{m=1}^{n}\sum_{w=1}^{\bar{w}}\sum_{\alpha\in\mathcal{V}}\|z_{\alpha}^{m,w}-z_{\alpha}^{m,w-1}\|+\sum_{\alpha\in\mathcal{V}}\|z_{\alpha}^{1,0}\|\nonumber \\
 & \overset{\eqref{eq:z-bdd-for-E}}{\leq} & M\|v_{H}^{1,0}\|+\sum_{\alpha\in\mathcal{V}}\|z_{\alpha}^{1,0}\|+\sum_{w=1}^{\bar{w}}\left(\sum_{\alpha\in\mathcal{E}}\|z_{\alpha}^{n,w}-z_{\alpha}^{n,w-1}\|\right)\nonumber \\
 &  & +\sum_{m=1}^{n-1}\sum_{w=1}^{\bar{w}}\left(M\|v_{H}^{m,w}-v_{H}^{m,w-1}\|+\sum_{\alpha\in\mathcal{V}}\|z_{\alpha}^{m,w}-z_{\alpha}^{m,w-1}\|\right)\nonumber \\
 & \overset{\eqref{eq:all-3-bdd}}{\leq} & M\|v_{H}^{1,0}\|+\sum_{\alpha\in\mathcal{V}}\|z_{\alpha}^{1,0}\|+MC_{2}\sum_{m=1}^{n}\sum_{w=1}^{\bar{w}}\|v_{A}^{n,w}-v_{A}^{n,w-1}\|.\nonumber 
\end{eqnarray}
By the Cauchy Schwarz inequality, we have 
\begin{equation}
\sum_{m=1}^{n}\sum_{w=1}^{\bar{w}}\|v_{A}^{m,w}-v_{A}^{m,w-1}\|\leq\sqrt{n\bar{w}}\sqrt{\sum_{m=1}^{n}\sum_{w=1}^{\bar{w}}\|v_{A}^{m,w}-v_{A}^{m,w-1}\|^{2}}.\label{eq:sum-bdd-by-sqrt-n}
\end{equation}
Since the second square root of the right hand side of \eqref{eq:sum-bdd-by-sqrt-n}
is bounded by Theorem \ref{thm:convergence}(i), we make use of \eqref{eq:2nd-big-ineq}
to obtain the conclusion \eqref{eq:sqrt-growth-sum-z} as needed. 
\end{proof}
\begin{remark}

(Convergence rate) An aspect of Algorithm \ref{alg:Ext-Dyk} that
we do not cover in this paper is the convergence rate. In the case
where there are no dual minimizers, components of the dual variables
$\{\{z_{\alpha}^{n,w}\}_{\alpha\in\mathcal{E}\cup\mathcal{V}}\}_{{1\leq n<\infty\atop 0\leq w\leq\bar{w}}}$
need not be bounded. But in the case where the variables $\{\{z_{\alpha}^{n,w}\}_{\alpha\in\mathcal{E}\cup\mathcal{V}}\}_{{1\leq n<\infty\atop 0\leq w\leq\bar{w}}}$
remain bounded as $n\to\infty$, an $O(1/n)$ rate was shown for the
dual objective function, which leads to an $O(1/\sqrt{n})$ convergence
rate of the distance $\|x^{n,w}-x^{*}\|$ to the optimal solution
$x^{*}$. The ideas for these results are presented in \cite[Section 3]{Pang_Dyk_spl}.
Such ideas were already present in \cite{Beck_Tetruashvili_2013,Beck_alt_min_SIOPT_2015}
for example. 

\end{remark}

\subsection{Generality of \eqref{eq:Dyk-primal}}

Another case of interest is when $\frac{1}{2}\|x-x_{0}\|^{2}$ in
\eqref{eq:Dyk-primal} is replaced by $\frac{1}{2}\|x-x_{0}\|_{Q}^{2}$,
where where $\|x\|_{Q}^{2}=\langle x,Qx\rangle$ and $Q$ is a block
diagonal positive definite matrix. In the case where $Q$ is such
that $\|x\|_{Q}^{2}=\sum_{i\in\mathcal{V}}\lambda_{i}\|x_{i}\|^{2}$
for some $\lambda\in\mathbb{R}^{|\mathcal{V}|}$ such that $\lambda>0$
and $\mathbf{f}_{i}\equiv0$ for all $i\in\mathcal{V}$, then the
minimizer of \eqref{eq:Dyk-primal} is $\frac{1}{\sum_{i\in\mathcal{V}}\lambda_{i}}\sum_{i\in\mathcal{V}}\lambda_{i}y_{i}^{0}$.
In other words, \eqref{eq:Dyk-primal} becomes a weighted average
consensus problem. 

We show how to transform a problem involving $\frac{1}{2}\|x-x_{0}\|_{Q}^{2}$
to one involving $\frac{1}{2}\|x-x_{0}\|^{2}$. Note that 
\begin{eqnarray*}
 &  & \min_{x\in X^{|\mathcal{V}|}}\frac{1}{2}\|x-x_{0}\|_{Q}^{2}+\sum_{(i,j)\in\mathcal{E}}\delta_{H_{(i,j)}}(x)+\sum_{i\in\mathcal{V}}\mathbf{f}_{i}(x)\\
 & \equiv & \min_{x\in X^{|\mathcal{V}|}}\frac{1}{2}\|Q^{1/2}x-Q^{1/2}x_{0}\|^{2}+\sum_{(i,j)\in\mathcal{E}}\delta_{Q^{1/2}H_{(i,j)}}(Q^{1/2}x)+\sum_{i\in\mathcal{V}}\mathbf{f}_{i}\circ Q^{-1/2}(Q^{1/2}x).
\end{eqnarray*}
We can thus let $v^{0}$ be $Q^{1/2}x_{0}$, and seek the variable
$v=Q^{1/2}x$. The function $\delta_{Q^{1/2}H_{(i,j)}}(\cdot)$ requires
a transformation of the set \eqref{eq:H-i-j-subspace}, but the transformed
problem would fit the framework of Dykstra's algorithm. 

\section{\label{sec:non-strongly-convex-alg}Distributed algorithm for functions
not strongly convex}

We saw earlier that the minimization of the sum of strongly convex
functions can be minimized over a network. A natural question to ask
is whether it is possible to minimize the sum of functions that are
not necessarily strongly convex in the same setting. 

A technique for minimizing \eqref{eq:better-dual} is to choose 2
or more nodes , say $S_{k}\subset\mathcal{V}$ (which preferably forms
a connected subgraph to allow for communications), and then minimize
the function varying only the dual variables corresponding to the
chosen nodes. This leads to Algorithm \ref{alg:distributed-dual-ascent}. 

\begin{algorithm}[h]
\caption{Distributed dual ascent algorithm}
\label{alg:distributed-dual-ascent}

Consider the problem \eqref{eq:better-dual}. 

Let $\{y_{i}^{0}\}_{i\in\mathcal{V}}\subset X$ be starting variables
such that 
\[
\sum_{i\in\mathcal{V}}y_{i}^{0}=0.
\]

For $n=1,\dots$

$\quad$Find a set $S_{n}\subset\mathcal{V}$ such that $|S_{n}|\geq2$. 

$\quad$For all $i\in S_{n}$, define $y_{i}^{n}$ so that 
\begin{eqnarray*}
\{y_{i}^{n}\}_{i\in S_{n}}\in & \underset{\{y_{i}\}_{i\in S_{n}}}{\arg\max} & -\sum_{i\in S_{n}}f_{i}^{*}(y_{i})\\
 & \mbox{s.t. } & \sum_{i\in S_{n}}y_{i}=\sum_{i\in S_{n}}y_{i}^{n-1}.
\end{eqnarray*}

$\quad$Define $y_{i}^{k}=y_{i}^{k-1}$ for all $i\notin S_{n}$. 

End for
\end{algorithm}
Such a method is analogous to the method of alternating minimization,
which have stationary points that are not optimal points. We now show
an example of such a stationary point for Algorithm \ref{alg:distributed-dual-ascent}. 

\begin{example}

\label{exa:alg-is-stuck} (Algorithm \ref{alg:distributed-dual-ascent}
can get stuck at non-optimal value) Consider a graph with $\mathcal{V}=\{1,2,3\}$
and $\mathcal{E}=\{(1,2),(2,3)\}$. Let $f_{i}:\mathbb{R}\to\bar{\mathbb{R}}$,
$i\in\{1,2,3\}$, be defined by 
\[
f_{1}(x)=\frac{1}{2}(x+1)^{2},\quad f_{2}(x)=\delta_{\{0\}}(x),\quad f_{3}(x)=\frac{1}{2}(x-1)^{2},
\]
and we have $f_{1}^{*}(z)=\frac{1}{2}z^{2}-z$, $f_{2}^{*}(z)=\delta_{\mathbb{R}}(z)$
and $f_{3}^{*}(z)=\frac{1}{2}z^{2}+z$. Let $\bar{y}=(0,0,0)\in\mathbb{R}^{3}$.
For $S\subset\{1,2,3\}$, denote the dual problem $(DP_{S})$ by 
\begin{eqnarray*}
(DP_{S}) & \max_{y_{i}\in\mathbb{R},i\in S} & -\sum_{i\in S}f_{i}^{*}(y_{i})\\
 & \mbox{s.t.} & \sum_{i\in S}y_{i}=\sum_{i\in S}\bar{y}_{i}.
\end{eqnarray*}
The problem $(DP_{\{1,2\}})$ has minimizer $y_{1}=y_{2}=0$ with
$-1\in\partial f_{1}^{*}(0)$ and $-1\in\partial f_{2}^{*}(0)$, and
the problem $(DP_{\{2,3\}})$ minimizer $y_{2}=y_{3}=0$ with $1\in\partial f_{2}^{*}(0)$
and $1\in\partial f_{3}^{*}(0)$. Hence the problem $(DP_{\{1,2,3\}})$
has a stationary point of $\bar{y}=(0,0,0)$ for the alternating minimization
method. However, the global minimizer to $(DP_{\{1,2,3\}})$ is $(1,0,-1)$
with $0\in\partial f_{i}^{*}(y_{i})$ for $i\in\{1,2,3\}$. 

\end{example}

Note from this example that the failure can be identified from the
fact that $x=-1$ if we use the edge $(1,2)$ and $x=1$ if we use
the edge $(2,3)$. Node 2 should be able to figure out that the $x$
values corresponding to edges $(1,2)$ and $(2,3)$ are too far apart,
and one needs to minimize $f_{1}^{*}(\cdot)+f_{3}^{*}(\cdot)$ in
order to avoid convergence to a non optimal value. 

We now give a proof for the convergence of Algorithm \ref{alg:distributed-dual-ascent},
which is based on the proof in \cite{Tseng_JOTA_2001} (and who in
turn cited other references). To shorten notation, for each set $S\subset\mathcal{V}$,
we let $\mathcal{D}(S)$ be the set of directions $d$ in $X^{|\mathcal{V}|}$
defined by 
\[
\mathcal{D}(S):=\left\{ d\in X^{|\mathcal{V}|}:\sum_{i=1}^{|\mathcal{V}|}d_{i}=0\mbox{ and supp}(d)\subset S\right\} ,
\]
where supp$(d)$ is the set $\{i:d_{i}\neq0\}$. We define $G:X^{|\mathcal{V}|}\to\bar{\mathbb{R}}$
by 
\[
G(y)=\sum_{i\in\mathcal{V}}f_{i}^{*}(y_{i}),
\]
and let $G'(y;d)$ be the directional derivative of $G$ at $y$ in
the direction $d$. 
\begin{theorem}
\label{thm:Conv-general-alg}(Convergence of Algorithm \ref{alg:distributed-dual-ascent})
Suppose that there is an integer $T$ such that for every $n\geq1$,
the sets $\{S_{n+i}\}_{i=1}^{T}$ satisfies the following:

\begin{enumerate}
\item [(a)]Suppose $S',S''$ are elements in $\{S_{n+i}\}_{i=1}^{T}$ such
that $S'\cap S''\neq\emptyset$ and $y\in X^{|\mathcal{V}|}$. Then
for all $d\in\mathcal{D}(S'\cup S'')$, we can find $d'\in\mathcal{D}(S')$
and $d''\in\mathcal{D}(S'')$ such that $d=d'+d''$ and 
\begin{equation}
G'(y;d)=G'(y;d')+G'(y;d'').\label{eq:dirn-derv-condn}
\end{equation}
\item [(b)]Suppose $y\in X^{|\mathcal{V}|}$. If for all $r\in\{1,\dots,T\}$,
$G'(y;d_{r})\geq0$ for all $d_{r}\in\mathcal{D}(S_{n+r})$, then
$y$ is a minimizer of $G(\cdot)$. 
\end{enumerate}
Suppose further that the sequence $\{y^{n}\}_{n=1}^{\infty}$ is bounded.
Then every cluster point of $\{y^{n}\}_{n=1}^{\infty}$ is a minimizer
of $G$. 
\end{theorem}

\begin{proof}
Our proof is adapted from the ideas in \cite[Section 4]{Tseng_JOTA_2001}.
Suppose $\bar{y}\in X^{|\mathcal{V}|}$ is a cluster point of $\{y^{n}\}_{n=1}^{\infty}$
and that $\mathcal{R}\subset\{1,2,\dots\}$ is such that $\lim_{r\in\mathcal{R}}y^{r}=\bar{y}$.
We can assume, by taking subsequences if necessary, that $\lim_{r\in\mathcal{R}}y^{r-T+1+j}$
converges to some $\bar{y}^{j}$ for all $j\in\{1,\dots,T\}$. We
have $\bar{y}^{T-1}=\bar{y}$. We also note that $\{G(y^{n})\}_{n=1}^{\infty}$
is a non-increasing sequence, so 
\begin{equation}
\lim_{r\in\mathcal{R}}G(y^{r-T+1+j})\mbox{ exists for all }j\in\{1,\dots,T\}.\label{eq:Tseng-6}
\end{equation}

Next, we can assume that for each $j\in\{1,\dots,T\}$, the set $\{S_{r-T+1+j}\}_{r\in\mathcal{R}}$
depends only on $j$, which we call $\bar{S}_{j}$. For each $j\in\{1,\dots,T\}$,
since $\bar{S}_{j}$ is chosen at iteration $r-T+1+j$ for $r\in\mathcal{R}$,
we have 
\begin{eqnarray*}
G(y^{r-T+1+j}) & \leq & G(y^{r-T+1+j}+d_{j})\mbox{ for all }d_{j}\in\mathcal{D}(\bar{S}_{j})\\
y_{i}^{r-T+j} & = & y_{i}^{r-T+1+j}\mbox{ for all }i\notin\bar{S}_{j}.
\end{eqnarray*}
Then the continuity of $G(\cdot)$ gives us 
\begin{eqnarray}
G(\bar{y}^{j}) & \leq & G(\bar{y}^{j}+d_{j})\mbox{ for all }d_{j}\in\mathcal{D}(\bar{S}_{j})\label{eq:Tseng-7}\\
\bar{y}_{i}^{j-1} & = & \bar{y}_{i}^{j}\mbox{ for all }i\notin\bar{S}_{j}.\nonumber 
\end{eqnarray}
We have $G(\bar{y}^{1})=\cdots=G(\bar{y}^{T})$. The previous line
also gives $\bar{y}^{j}-\bar{y}^{j-1}\in\mathcal{D}(\bar{S}_{j})$,
so 
\begin{equation}
G(\bar{y}^{j-1})=G(\bar{y}^{j})\overset{\eqref{eq:Tseng-7}}{\leq}G(\bar{y}^{j}+(d_{j}+\bar{y}^{j-1}-\bar{y}^{j}))=G(\bar{y}^{j-1}+d_{j})\mbox{ for all }d_{j}\in\mathcal{D}(\bar{S}_{j}).\label{eq:Tseng-8}
\end{equation}
We claim that for $j=1,\dots,T-1$, 
\begin{equation}
G(\bar{y}^{j})\leq G(\bar{y}^{j}+d_{k})\mbox{ for all }d_{k}\in\mathcal{D}(\bar{S}_{1}),\mathcal{D}(\bar{S}_{2}),\dots,\mbox{ or }\mathcal{D}(\bar{S}_{j}).\label{eq:Tseng-9}
\end{equation}
By \eqref{eq:Tseng-7}, \eqref{eq:Tseng-9} holds for $j=1$. Suppose
\eqref{eq:Tseng-9} holds for $j=1,\dots,l-1$ for some $l\in\{2,\dots,T-1\}$.
We show that \eqref{eq:Tseng-9} holds for $j=l$. From \eqref{eq:Tseng-8},
\[
G(\bar{y}^{l-1})\le G(\bar{y}^{l-1}+d_{l})\mbox{ for all }d_{l}\in\mathcal{D}(\bar{S}_{l}),
\]
implying that 
\begin{equation}
G'(\bar{y}^{l-1};\bar{y}^{l}-\bar{y}^{l-1}+v)\geq0\mbox{ for all }v\in\mathcal{D}(\bar{S}_{l}).\label{eq:crit-1}
\end{equation}
Also, since \eqref{eq:Tseng-9} holds for $j=l-1$, we have, for all
$k\in\{1,\dots,l-1\}$, 
\[
G(\bar{y}^{l-1})\leq G(\bar{y}^{l-1}+d_{k}-v)\mbox{ for all }d_{k}\in\mathcal{D}(\bar{S}_{k})\mbox{ and }v\in\mathcal{D}(\bar{S}_{l}\cap\bar{S}_{k}),
\]
which in turn implies 
\begin{equation}
G'(\bar{y}^{l-1};d_{k}-v)\geq0\mbox{ for all }d_{k}\in\mathcal{D}(\bar{S}_{k})\mbox{ and }v\in\mathcal{D}(\bar{S}_{l}\cap\bar{S}_{k}).\label{eq:crit-2}
\end{equation}
If $\bar{S}_{l}\cap\bar{S}_{k}=\emptyset$, then $v$ can be taken
to be zero, and we get \eqref{eq:dirn-derv-condn}. By \eqref{eq:crit-1}
and \eqref{eq:crit-2} and property (a) for each $d_{k}$, we can
choose $v$ such that for all $d_{k}\in\mathcal{D}(\bar{S}_{k})$
\begin{equation}
G'(\bar{y}^{l-1};\bar{y}^{l}-\bar{y}^{l-1}+d_{k})\overset{\scriptsize{\mbox{ppty (a)}}}{=}G'(\bar{y}^{l-1};\bar{y}^{l}-\bar{y}^{l-1}+v)+G'(\bar{y}^{l-1};d_{k}-v)\overset{\scriptsize{\eqref{eq:crit-1},\eqref{eq:crit-2}}}{\geq}0.\label{eq:dirn-and-step}
\end{equation}
Since $G(\cdot)$ is convex, 
\[
G(\bar{y}^{l}+d_{k})=G\big(\bar{y}^{l-1}+(\bar{y}^{l}-\bar{y}^{l-1}+d_{k})\big)\overset{\eqref{eq:dirn-and-step}}{\geq}G(\bar{y}^{l-1})=G(\bar{y}^{l})\mbox{ for all }d_{k}\in\mathcal{D}(\bar{S}_{k}).
\]
Since \eqref{eq:Tseng-7} holds with $j=l$, \eqref{eq:Tseng-9} holds
for $j=l$. So \eqref{eq:Tseng-9} holds for all $j\in\{1,\dots,T-1\}$.
Taking $j=T-1$ for \eqref{eq:Tseng-9} and combining property (b)
proves that $\bar{y}=\bar{y}^{T-1}$ is a minimizer of $G(\cdot)$. 
\end{proof}
Define $\mathcal{V}_{sm}\subset\mathcal{V}$ to be such that 
\[
\mathcal{V}_{sm}:=\{i\in\mathcal{V}:f_{i}^{*}(\cdot)\mbox{ is smooth}\}.
\]

We give more insight on Properties (a) and (b) in Theorem \ref{thm:Conv-general-alg}.
\begin{proposition}
Property (a) in Theorem \ref{thm:Conv-general-alg} is satisfied if
for any two elements $S',S''$ in $\{S_{n+i}\}_{i=1}^{T}$, either
$S'\cap S''=\emptyset$ or $S'\cap S''\cap\mathcal{V}_{sm}\neq\emptyset$. 
\end{proposition}

\begin{proof}
We only need to consider the case when $S'\cap S''\cap\mathcal{V}_{sm}\neq\emptyset$.
We want to show that if $d\in\mathcal{D}(S'\cup S'')$, then $d$
can be written as $d=d'+d''$, where $d'\in\mathcal{D}(S')$ and $d''\in\mathcal{D}(S'')$,
so that \eqref{eq:dirn-derv-condn} holds. If $S'\cap S''\cap\mathcal{V}_{sm}\neq\emptyset$,
then let $\bar{i}\in S'\cap S''\cap\mathcal{V}_{sm}$. For a given
$d\in\mathcal{D}(S'\cup S'')$, define $d'$ and $d''$ so that 
\[
d'_{i}=\begin{cases}
d_{i} & \mbox{ if }i\in S'\backslash\{\bar{i}\}\\
-\sum_{i\in S'\backslash\{\bar{i}\}}d_{i} & \mbox{ if }i=\bar{i}\\
0 & \mbox{ otherwise }
\end{cases}\mbox{ and }d''_{i}=\begin{cases}
d_{i} & \mbox{ if }i\in S''\backslash S'\\
-\sum_{i\in S''\backslash S'}d_{i} & \mbox{ if }i=\bar{i}\\
0 & \mbox{ otherwise.}
\end{cases}
\]
It is clear to see that $d=d'+d''$, $d'\in\mathcal{D}(S')$ and $d''\in\mathcal{D}(S'')$.
From the smoothness of $f_{\bar{i}}^{*}(\cdot)$, we have $[f_{\bar{i}}^{*}]'(x;d_{\bar{i}}'+d_{\bar{i}}'')=[f_{\bar{i}}^{*}]'(x;d_{\bar{i}}')+[f_{\bar{i}}^{*}]'(x;d_{\bar{i}}'')$,
which gives \eqref{eq:dirn-derv-condn} as needed.
\end{proof}
\begin{proposition}
Property (b) in Theorem \ref{thm:Conv-general-alg} is satisfied if

\begin{enumerate}
\item For all $n\geq0$ and $y$, the condition $G'(y;d)\geq0$ for all
$d\in\mathcal{D}(S_{n})$ implies the existence of KKT multipliers
of 
\begin{eqnarray*}
 & \max_{y'_{i}\in X,i\in S_{n}} & -\sum_{i\in S_{n}}f_{i}^{*}(y'_{i})\\
 & \mbox{s.t.} & \sum_{i\in S_{n}}y'_{i}=\sum_{i\in S_{n}}y_{i}
\end{eqnarray*}
at a maximizer $\bar{y}$. Specifically, there exists $x\in X$ such
that $x\in\partial f_{i}^{*}(\bar{y}_{i})$ for all $i\in S_{n}$.
\item For every $\bar{i},\bar{j}\in\mathcal{V}$, we can find a sequence
of sets $\{\tilde{S}_{k}\}_{k=1}^{K}\subset\{S_{n+r}\}_{r=1}^{T}$
such that $\bar{i}\in\tilde{S}_{1}$, $\bar{j}\in\tilde{S}_{K}$,
and $\tilde{S}_{k}\cap\tilde{S}_{k+1}\cap\mathcal{V}_{sm}\neq\emptyset$
for all $k\in\{1,\dots,K-1\}$. 
\end{enumerate}
\end{proposition}

\begin{proof}
Recall the $y$ in property (b) in Theorem \ref{thm:Conv-general-alg}.
Through condition (2), it suffices to prove that if $\tilde{S}_{1}$
and $\tilde{S}_{2}$ are such that $\tilde{S}_{1}\cap\tilde{S}_{2}\cap\mathcal{V}_{sm}\neq\emptyset$,
then there exists $x$ such that $x\in\partial f_{i}^{*}(y_{i})$
for all $i\in\tilde{S}_{1}\cup\tilde{S}_{2}$, which is in turn easy
from condition (1). 
\end{proof}
In Example \ref{exa:alg-is-stuck}, we see that $2\notin\mathcal{V}_{sm}$,
so Theorem \ref{thm:Conv-general-alg} does not apply.

\subsection{Connection between Sections \ref{sec:conv-Dyk} and \ref{sec:non-strongly-convex-alg}}

We now give a connection between the algorithms in the Sections \ref{sec:conv-Dyk}
and \ref{sec:non-strongly-convex-alg}. For a graph $(\mathcal{V},\mathcal{E})$,
construct the graph $(\mathcal{V}^{+},\mathcal{E}^{+})$ via 
\begin{eqnarray*}
\mathcal{V}^{+} & = & \mathcal{V}\times\{0,1\},\\
\mathcal{E}^{+} & = & \big\{\big((i,0),(j,0)\big):(i,j)\in\mathcal{E}\big\}\cup\big\{\big((i,0),(i,1)\big):i\in\mathcal{V}\big\}.
\end{eqnarray*}
One can easily check that $|\mathcal{V}^{+}|=2|\mathcal{V}|$ and
$|\mathcal{E}^{+}|=|\mathcal{E}|+|\mathcal{V}|$. Let the function
associated with the vertex $(i,s)$ be $f_{i,s}:X\to\bar{\mathbb{R}}$
defined by 
\[
f_{(i,s)}(x)=\begin{cases}
\frac{1}{2}\|x-[x_{0}]_{i}\|^{2} & \mbox{ if }s=0\\
f_{i}(x) & \mbox{\,if }s=1.
\end{cases}
\]
Note that $\min_{x\in X}\sum_{(i,s)\in\mathcal{V}^{+}}f_{(i,s)}(x)$
is equivalent to the problem \eqref{eq:separate_obj} considered in
Dykstra's algorithm. 

The dual problem 
\[
\max_{y_{(i,s)}\in X,(i,s)\in\mathcal{V}^{+}}\left\{ -\sum_{(i,s)\in\mathcal{V}^{+}}f_{(i,s)}^{*}(y_{(i,s)}):\sum_{(i,s)\in\mathcal{V}^{+}}y_{(i,s)}=0\right\} 
\]
 (recall how \eqref{eq:better-dual} is derived as the dual of \eqref{eq:centralized-primal})
can be simplified to be 
\begin{eqnarray}
 & \underset{y_{i,s}\in X:i\in\mathcal{V},s\in\{0,1\}}{\max} & -\sum_{i\in\mathcal{V}}f_{i}^{*}(y_{i,1})-\sum_{i\in\mathcal{V}}\left[\frac{1}{2}\|y_{i,0}+[x_{0}]_{i}\|^{2}-\frac{1}{2}\|[x_{0}]_{i}\|^{2}\right]\label{eq:first-common-dual}\\
 & \mbox{s.t.} & \sum_{i\in\mathcal{V}}(y_{i,0}+y_{i,1})=0.\nonumber 
\end{eqnarray}
Recall the dual problem in \eqref{eq:dual-fn} and \eqref{eq:Dykstra-dual-defn}.
Define the variable $z_{e}\in X^{|\mathcal{V}|}$ to be 
\[
z_{e}:=\sum_{(i,j)\in\mathcal{E}}z_{(i,j)}.
\]
Suppose $z_{(i,j)}\in H_{(i,j)}^{\perp}$. Then $\delta_{H_{(i,j)}}^{*}(z_{(i,j)})=0$.
Also, $z_{(i,j)}\in H_{(i,j)}^{\perp}\subset D^{\perp}$, so $z_{e}\in D^{\perp}$.
The dual problem in \eqref{eq:dual-fn} and \eqref{eq:Dykstra-dual-defn}
becomes 
\begin{eqnarray}
 & \underset{\{z_{e}\in D^{\perp}\}\cup\{z_{\alpha}\in X^{|\mathcal{V}|}:\alpha\in\mathcal{V}\}}{\max} & -\sum_{i\in\mathcal{V}}\mathbf{f}_{i}^{*}(z_{i})-\frac{1}{2}\left\Vert \sum_{i\in\mathcal{V}}z_{i}+z_{e}-x_{0}\right\Vert ^{2}+\frac{1}{2}\|x_{0}\|^{2}.\label{eq:second-common-dual}
\end{eqnarray}
We now show how \eqref{eq:first-common-dual} and \eqref{eq:second-common-dual}
are related. 
\begin{proposition}
\label{prop:1-1-corr}Consider the problems \eqref{eq:first-common-dual}
and \eqref{eq:second-common-dual}.

\begin{enumerate}
\item The two problems are related through a change of variables. Specifically,
if $\{y_{(i,s)}\}_{(i,s)\in\mathcal{V}^{+}}\subset X$ were obtained
from $\{z_{i}\}_{i\in\mathcal{V}}\subset X^{|\mathcal{V}|}$ and $z_{e}\in X^{|\mathcal{V}|}$
by 
\[
y_{i,1}=[z_{i}]_{i}\mbox{ and }y_{i,0}=-[z_{e}]_{i}-[z_{i}]_{i}\mbox{ for all }i\in\mathcal{V},
\]
then the objective values in \eqref{eq:second-common-dual} and \eqref{eq:first-common-dual}
coincide. Conversely, if $\{z_{i}\}_{i\in\mathcal{V}}$ and $z_{e}$
were obtained from $\{y_{(i,s)}\}_{(i,s)\in\mathcal{V}^{+}}$ by 
\[
[z_{i}]_{j}=\begin{cases}
0 & \mbox{ if }j\neq i\\
y_{i,1} & \mbox{ if }j=i
\end{cases}\mbox{ and }[z_{e}]_{i}=-y_{i,0}-y_{i,1}\mbox{ for all }i,j\in\mathcal{V},
\]
then $z_{e}\in D^{\perp}$ and the objective values in \eqref{eq:second-common-dual}
and \eqref{eq:first-common-dual} coincide.
\item Let $S^{1}\subset\mathcal{V}^{+}$ be a connected subset of vertices
in the graph $(\mathcal{V}^{+},\mathcal{E}^{+})$ so that $|S^{1}|>1$.
Define $\Pi_{0}S^{1}\subset\mathcal{V}$ to be the set 
\[
\Pi_{0}S^{1}:=\{i\in\mathcal{V}:(i,0)\in S^{1}\}.
\]
Let $\Pi_{1}S^{1}$ be similarly defined. With respect to the graph
$(\mathcal{V},\mathcal{E})$, suppose that there is a subset $\mathcal{E}'$
of $\mathcal{E}$ not containing any cycles that connects all the
vertices in $\Pi_{0}S^{1}$. Since $S^{1}$ is a subset of connected
vertices in the graph $(\mathcal{V}^{+},\mathcal{E}^{+})$, we have
$\Pi_{1}S^{1}\subset\Pi_{0}S^{1}$. \\
Suppose that for a fixed $\{y_{(i,s)}\}_{(i,s)\in\mathcal{V}^{+}}\subset X$,
a subproblem of \eqref{eq:first-common-dual} is solved with only
variables $y_{(i,s)}$ indexed by $S^{1}$ allowed to vary while the
other variables stay fixed. Then under the change of variables in
(1), this subproblem is equivalent to solving the subproblem in \eqref{eq:second-common-dual}
where 

\begin{enumerate}
\item $[z_{e}]_{i}$ is allowed to vary if and only if $i$ is an endpoint
of some edge in $\mathcal{E}'$, and 
\item $z_{i}$ is allowed to vary if and only if $i\in\Pi_{1}S^{1}$. 
\end{enumerate}
\end{enumerate}
\end{proposition}

\begin{proof}
Statement 1 is obvious from the constructions. 

We now work on Statement 2. From $|S^{1}|>1$ and the definition of
$(\mathcal{V}^{+},\mathcal{E}^{+})$, if $(i,1)\in\mathcal{V}^{+}$,
then $(i,0)\in\mathcal{V}^{+}$. So for each $i\in\mathcal{V}$, there
are three cases: (1) $(i,1)\notin\mathcal{V}^{+}$ and $(i,0)\notin\mathcal{V}^{+}$
(in which case there is nothing to do), (2) $(i,0)\in\mathcal{V}^{+}$
and $(i,1)\notin\mathcal{V}^{+}$ and (3) $(i,0)\in\mathcal{V}^{+}$
and $(i,1)\in\mathcal{V}^{+}$. 

In case (2), if the term $y_{i,0}$ is in $S^{1}$, then $y_{i,0}$
affects only the $\frac{1}{2}\|y_{i,0}+[x_{0}]_{i}\|^{2}$ in \eqref{eq:first-common-dual}.
In turn, $[z_{e}]_{i}$ only affects $\frac{1}{2}\|[z_{e}]_{i}+[z_{i}]_{i}-[x_{0}]_{i}\|^{2}$
in the quadratic term in \eqref{eq:second-common-dual}. 

In case (3), it is clear that the term $f_{i}(y_{i,1})$ varies through
$y_{i,1}$ if and only if $\mathbf{f}_{i}(z_{i})$ varies through
$z_{i}$. Recall that $\Pi_{0}S^{1}\subset\Pi_{1}S^{1}$, so if the
term $y_{i,1}$ is in $S^{1}$, then $y_{i,0}$ is in $S^{1}$. The
terms $y_{i,1}$ and $y_{i,0}$ then combine to affect $f_{i}^{*}(y_{i,1})+\frac{1}{2}\|y_{i,0}+[x_{0}]_{i}\|^{2}$.
Correspondingly, $[z_{e}]_{i}$ and $[z_{i}]_{i}$ combine to affect
$\mathbf{f}_{i}^{*}(z_{i})+\frac{1}{2}\|[z_{e}]_{i}+[z_{i}]_{i}-[x_{0}]_{i}\|^{2}$.

To wrap up, note that the constraint in \eqref{eq:first-common-dual}
corresponds to $z_{e}\in D^{\perp}$. 
\end{proof}
One can easily figure out that (2a) in Proposition \ref{prop:1-1-corr}
corresponds to varying $z_{(i,j)}$ for all $(i,j)\in\mathcal{E}'$
in the original dual problem of \eqref{eq:dual-fn} and \eqref{eq:Dykstra-dual-defn}.

\section{\label{sec:Accelerate}Accelerated methods for \eqref{eq:dual-fn}}

In this section, we write down an accelerated proximal gradient (APG)
algorithm \cite{Nesterov_1983,BeckTeboulle2009,Tseng_APG_2008} on
the dual problem described through \eqref{eq:dual-fn} and \eqref{eq:Dykstra-dual-defn}
that allows for greedy steps that can be performed asynchronously.
Before we continue, we remark that we had shown that an APG with greedy
steps can be performed on the formulation \eqref{eq:illus-dyk-D}
in \cite{Pang_DBAP}. We point out that the APG derived from \eqref{eq:dual-fn}
and \eqref{eq:Dykstra-dual-defn} has a much lower Lipschitz constant
and allows for greedy steps of the form \eqref{eq:Dykstra-min-subpblm}. 

We first recall a variant of the accelerated proximal gradient in
\cite{Tseng_APG_2008}. In view of the clash of variables, we substitute
the variables $x$, $y$ and $z$ in \cite{Tseng_APG_2008} to be
$u$, $v$ and $w$, and then substitute the $f^{P}(\cdot)$ in \cite{Tseng_APG_2008}
for the function $-F(\cdot)$ in \eqref{eq:Dykstra-dual-defn}. (Their
algorithm includes allowing for the domain of the optimization problem
for $w^{k+1}$ to change in each iteration, which we omit.) Let 
\begin{align}
l(u,v) & =\underbrace{\frac{1}{2}\left\Vert x_{0}-\sum_{\alpha\in\mathcal{E}\cup\mathcal{V}}v_{\alpha}\right\Vert ^{2}-\frac{1}{2}\|x_{0}\|^{2}}\label{eq:l-u-v}\\
 & \phantom{=}+\underbrace{\left\langle \left(\sum_{\alpha\in\mathcal{E}\cup\mathcal{V}}v_{\alpha}\right)-x_{0},\sum_{\alpha\in\mathcal{E}\cup\mathcal{V}}(u_{\alpha}-v_{\alpha})\right\rangle }+\underbrace{\sum_{\alpha\in\mathcal{E}\cup\mathcal{V}}h_{\alpha}(u_{\alpha})}\nonumber 
\end{align}

\begin{remark}

\label{rem:APG-details} We now explain the formula $l(u,v)$ above.
In \cite{Tseng_APG_2008}, the function that Algorithm \ref{alg:Tseng-2008}
aims to minimize was $f^{P}(u)=f(u)+P(u)$, where $f(\cdot)$ is smooth
and $P(\cdot)$ admits an easy calculation of its proximal, and has
a linearization $f(v)+\langle\nabla f(v),u-v\rangle+P(u)$. The underbraced
terms in \eqref{eq:l-u-v} play the role of the terms $f(v)$, $\langle\nabla f(v),u-v\rangle$
and $P(u)$ in the linearization of $-F(u)$ in \eqref{eq:Dykstra-dual-defn}. 

\end{remark}
\begin{algorithm}[H]
\caption{From \cite[Algorithm 1]{Tseng_APG_2008}}
\label{alg:Tseng-2008}

We want to find $u$ to minimize $-F(\cdot)$, which is the sum of
a convex smooth function and a convex separable function. Choose $\theta_{0}\in(0,1]$,
$u^{0}$, $w^{0}\in\mbox{dom}(P)$. Let $L>0$ be such that 
\begin{equation}
\begin{array}{c}
l(u,v)+\frac{L}{2}\|u-v\|^{2}\geq-F(u)\mbox{ for all }u,v\in[X^{|\mathcal{V}|}]^{|\mathcal{V}\cup\mathcal{E}|}.\end{array}\label{eq:show-cond}
\end{equation}
Go to 1.

1. Let 
\begin{eqnarray*}
v^{k} & = & (1-\theta_{k})u^{k}+\theta_{k}w^{k}\\
w^{k+1} & = & \underset{u\in X}{\arg\min}\{l(u;v^{k})+\theta_{k}L\frac{1}{2}\|u-w^{k}\|^{2})\},\\
\hat{u}^{k+1} & = & (1-\theta_{k})u^{k}+\theta_{k}w^{k+1},
\end{eqnarray*}

Choose $u_{k+1}$ to be such that 
\begin{equation}
\begin{array}{c}
-F(u^{k+1})\leq l(\hat{u}^{k+1};v^{k})+\frac{L}{2}\|\hat{u}^{k+1}-v^{k}\|^{2}.\end{array}\label{eq:Tseng-14}
\end{equation}

Choose $\theta_{k+1}\in(0,1]$ satisfying 
\[
\begin{array}{c}
\frac{1-\theta_{k+1}}{\theta_{k+1}^{2}}\leq\frac{1}{\theta_{k}^{2}}.\end{array}
\]

$k\leftarrow k+1$, and go to 1. 
\end{algorithm}
\begin{remark}

Algorithm \ref{alg:Tseng-2008} requires the condition \eqref{eq:show-cond}.
Since $l(\cdot,v)$ is the linearization of $-F(\cdot)$ at $v$,
we have $l(v,v)=-F(v)$. Since the smooth portions of both $l(u,v)+\frac{L}{2}\|u-v\|^{2}$
and $-F(\cdot)$ are quadratics, showing \eqref{eq:show-cond} is
equivalent to finding $L>0$ such that the Hessian of the smooth portion
of $l(u,v)+\frac{L}{2}\|u-v\|^{2}$ is greater than that of $-F(\cdot)$,
i.e., 
\begin{equation}
\frac{L}{2}\sum_{\alpha\in\mathcal{E}\cup\mathcal{V}}\|u_{\alpha}\|^{2}\geq\frac{1}{2}\left\Vert \sum_{\alpha\in\mathcal{E}\cup\mathcal{V}}u_{\alpha}\right\Vert ^{2}.\label{eq:choose-L}
\end{equation}

\end{remark}

Since the variables $\{u_{\alpha}\}_{\alpha\in\mathcal{E}\cup\mathcal{V}}$
are to satisfy the sparsity pattern in Proposition \ref{prop:sparsity},
we show that $L$ can be chosen as follows.
\begin{proposition}
(Choice of $L$ in Algorithm \ref{alg:Tseng-2008}) In order to satisfy
\eqref{eq:choose-L} while obeying the sparsity pattern in Proposition
\ref{prop:sparsity}, we can choose $L$ to be $\bar{d}+1$, where
$\bar{d}$ is the maximum degree of the vertices in the graph.
\end{proposition}

\begin{proof}
We look at the $i$-th component $\{[u_{\alpha}]_{i}\}_{\alpha\in\mathcal{E}\cup\mathcal{V}}$
of the terms in \eqref{eq:choose-L} for all $i\in\mathcal{V}$. As
long as we can prove that 
\begin{equation}
\frac{\bar{d}+1}{2}\sum_{\alpha\in\mathcal{E}\cup\mathcal{V}}\|[u_{\alpha}]_{i}\|^{2}\geq\frac{1}{2}\left\Vert \sum_{\alpha\in\mathcal{E}\cup\mathcal{V}}[u_{\alpha}]_{i}\right\Vert ^{2},\label{eq:choose-L-2}
\end{equation}
the conclusion will follow. In view of the sparsity pattern of the
$u$'s in Proposition \ref{prop:sparsity}, most of the $[u_{\alpha}]_{i}$'s
are zero. For all $i\in\mathcal{V}$, we define $\tilde{\mathcal{E}}_{i}$
to be the set of all edges $e\in\mathcal{E}$ such that one of the
endpoints is $i$. Then \eqref{eq:choose-L-2} reduces to 
\begin{equation}
\begin{array}{c}
\frac{\bar{d}+1}{2}\left(\|[u_{i}]_{i}\|^{2}+\underset{e\in\tilde{\mathcal{E}_{i}}}{\sum}\|[u_{e}]_{i}\|^{2}\right)\geq\frac{1}{2}\left\Vert [u_{i}]_{i}+\underset{e\in\tilde{\mathcal{E}_{i}}}{\sum}[u_{e}]_{i}\right\Vert ^{2},\end{array}\label{eq:choose-L-3}
\end{equation}
We form the vector $\tilde{u}\in X^{|\tilde{\mathcal{E}}_{i}|+1}$
so that it contains $[u_{i}]_{i}$ and $\{[u_{e}]_{i}\}_{e\in\tilde{\mathcal{E}}_{i}}$
as its components. The formula \eqref{eq:choose-L-3} can be seen
to be equivalent to 
\[
\begin{array}{c}
\tilde{u}^{T}\left[\left(\begin{array}{cccc}
(\bar{d}+1)I\\
 & (\bar{d}+1)I\\
 &  & \ddots\\
 &  &  & (\bar{d}+1)I
\end{array}\right)-\left(\begin{array}{cccc}
I & I & \cdots & I\\
I & I & \cdots & I\\
\vdots & \vdots & \ddots & \vdots\\
I & I & \cdots & I
\end{array}\right)\right]\tilde{u}\geq0,\end{array}
\]
which is clearly true. Thus we are done. 
\end{proof}
If we had used the formulation in \eqref{eq:Dykstra-min-subpblm}
without exploiting the sparsity in Proposition \ref{prop:sparsity},
then the corresponding $L$ would be $|\mathcal{V}\cup\mathcal{E}|$,
which is a much larger number than $\bar{d}+1$ in most large graphs.
Recall that $L$ was chosen so that \eqref{eq:show-cond} holds, and
$L$ should be as small as possible subject to this condition so that
the step for calculating $w^{k+1}$ in Algorithm \ref{alg:Tseng-2008}
would be minimizing a function closer to $-F(\cdot)$. This lower
value of $\bar{d}+1$ is one advantage of applying the APG on the
dual problem from \eqref{eq:dual-fn} and \eqref{eq:Dykstra-dual-defn}. 

We recall the convergence result of Algorithm \ref{alg:Tseng-2008}.
\begin{theorem}
\label{thm:Tseng-alg-conv}\cite[Corollary 1]{Tseng_APG_2008} (Convergence
of Algorithm \ref{alg:Tseng-2008}) Let 
\[
\{(u^{k},v^{k},w^{k},\theta_{k},X_{k})\}_{k}
\]
 be generated by Algorithm \ref{alg:Tseng-2008} with $\theta_{0}=1$.
For any $\epsilon>0$. Suppose $\theta_{k}\leq\frac{2}{k+2}$. Then
for any $u\in\mbox{dom}(P)$ with $-F(u)\leq\inf\,F+\epsilon$, we
have 
\[
\begin{array}{c}
\underset{i=0,1\dots,k+1}{\min}\{-F(u^{i})\}\leq-F(u)+\epsilon\mbox{ whenever }k\geq\sqrt{\frac{4L}{\epsilon}}\|x-x^{0}\|-2.\end{array}
\]
\end{theorem}

In the particular case where there is a minimizer $u^{*}$, Theorem
\ref{thm:Tseng-alg-conv} says that an $\epsilon$-optimal solution
for $-F(\cdot)$ is obtained if the number of iterations $k$ is the
Nesterov accelerated rate of $O(\sqrt{\frac{1}{\epsilon}})$ \cite{Nesterov_1983,BeckTeboulle2009,Tseng_APG_2008}.
In the case of Dykstra's algorithm (or block coordinate minimization),
the number of iterations needed to obtain an $\epsilon$-optimal solution
is typically $O(\frac{1}{\epsilon})$ (see for example \cite{Beck_Tetruashvili_2013,Beck_alt_min_SIOPT_2015}),
which is slower than the $O(\sqrt{\frac{1}{\epsilon}})$ rate. 

\begin{remark}

(On Theorem \ref{thm:Tseng-alg-conv}) The proof of Theorem \ref{thm:Tseng-alg-conv}
in \cite{Tseng_APG_2008} is for the algorithm with a modified \eqref{eq:Tseng-14},
with the left hand side being $l(u^{k+1};v^{k})+\frac{L}{2}\|u^{k+1}-v^{k}\|^{2}$
instead. But the proof in \cite{Tseng_APG_2008} carries over with
no changes at all.

\end{remark}

We now elaborate on how the greedy step can be applied to Algorithm
\ref{alg:Tseng-2008}.

\begin{remark}

(Greedy steps in Algorithm \ref{alg:Tseng-2008}) We remark that the
greedy step can be performed in \eqref{eq:Tseng-14}. Note that $u^{k+1}$
in \eqref{eq:Tseng-14} can be chosen to be $\hat{u}^{k+1}$. But
the greedy steps of the form \eqref{eq:Dykstra-min-subpblm} can be
performed in \eqref{eq:Tseng-14} (with $u$'s in place of $z$'s
there). These greedy steps can be performed asynchronously like as
discussed in Remark \ref{rem:distrib-comp}.

\end{remark}

\section{Conclusion}

We have done what we set out to do in Subsection \ref{subsec:contrib}.
In short, we noticed that a dual ascent algorithm can give us a distributed
and asynchronous algorithm with deterministic convergence for time-varying
graphs when the function on each vertex is strongly convex with a
known modulus. A separate related algorithm is proposed for the case
when the function on each vertex is not necessarily strongly convex,
but Example \ref{exa:alg-is-stuck} shows that the algorithm can fail
to converge.

Note that in Example \ref{exa:alg-is-stuck}, the failure of Algorithm
\ref{alg:distributed-dual-ascent} is identified by the primal variable
$-1$ in $\partial f_{1}^{*}(0)$ and $\partial f_{2}^{*}(0)$ being
obtained when $S=\{1,2\}$, while the primal variable $1$ in $\partial f_{2}^{*}(0)$
and $\partial f_{3}^{*}(0)$ was obtained when $S=\{2,3\}$. Is there
a primal dual algorithm that has a number of the properties listed
in Subsection \ref{subsec:Distrib-algs} and \ref{subsec:contrib}?

\section*{Acknowledgments}

We thank Mert G\"{u}rb\"{u}zbalaban
and Necdet Serhat Aybat for discussions leading to this paper. Lastly,
we thank the associate editor, the anonymous referees and the journal
staff for the quick review of this paper.

\bibliographystyle{siamplain}
\bibliography{refs}

\end{document}